\documentclass[11pt, letterpaper]{amsart}
\usepackage{amsfonts,amsmath, amsthm, amssymb, latexsym, epsfig}
\usepackage[all]{xy}
\usepackage{xspace}
\usepackage{comment}
\usepackage{setspace}
\usepackage{enumerate}
\usepackage[mathscr]{eucal}

\usepackage{bbold}

	\advance \topmargin by -\headheight
	\advance \topmargin by -\headsep

	\evensidemargin \oddsidemargin
	\newcommand{\ftn}[3]{ #1 : #2 \rightarrow #3 }
		\newcommand{\setof}[2]{\ensuremath{\left\{ #1 \: : \: #2 \right\}}}

		\newcommand{\ab}{\mathfrak{Ab}}

	\newcommand{\kl}{\ensuremath{\mathit{KL}}\xspace}
	\newcommand{\kk}{\ensuremath{\mathit{KK}}\xspace}

	\newcommand{\Hom}{\ensuremath{\operatorname{Hom}}}

	\newcommand{\id}{\ensuremath{\operatorname{id}}}

	\newcommand{\Ext}{\ensuremath{\operatorname{Ext}}}

	\newcommand{\multialg}[1]{\mathcal{M}(#1)\xspace}
	\newcommand{\corona}[1]{\mathcal{Q}(#1)\xspace}

		\newcommand{\Z}{\ensuremath{\mathbb{Z}}\xspace}
	\newcommand{\C}{\ensuremath{\mathbb{C}}\xspace}

	\newcommand{\N}{\ensuremath{\mathbb{N}}\xspace}
	
	\newcommand{\K}{\ensuremath{\mathbb{K}}\xspace}
	
	\newcommand{\catc}{\mathfrak{C}^{*}\text{-}\mathfrak{alg}}
	\newcommand{\kkc}{\mathfrak{KK}}

	\theoremstyle{plain}
	\newtheorem{thm}{Theorem}[section]
	\newtheorem{lemma}[thm]{Lemma}
	\newtheorem{theor}[thm]{Theorem}
	\newtheorem{propo}[thm]{Proposition}
	\newtheorem{corol}[thm]{Corollary}

	\theoremstyle{definition}
	\newtheorem{defin}[thm]{Definition}
	
	\newtheorem{remar}[thm]{Remark}
	\newtheorem{notat}[thm]{Notation}
	\newtheorem{examp}[thm]{Example}

	\numberwithin{equation}{section}
	\numberwithin{figure}{section}

\begin{document}
	\title{Strong classification of extensions of classifiable $C^{*}$-algebras}
	\author{S{\o}ren Eilers}
        \address{Department of Mathematical Sciences \\
        University of Copenhagen\\
        Universitetsparken~5 \\
        DK-2100 Copenhagen, Denmark}
        \email{eilers@math.ku.dk }
          \author{Gunnar Restorff}
\address{Faculty of Science and Technology\\University of Faroe 
Islands\\N\'oat\'un 3\\FO-100 T\'orshavn\\Faroe Islands}
\email{gunnarr@setur.fo}
	\author{Efren Ruiz}
        \address{Department of Mathematics\\University of Hawaii,
Hilo\\200 W. Kawili St.\\
Hilo, Hawaii\\
96720-4091 USA}
        \email{ruize@hawaii.edu}
        \date{\today}
	
%AMS info


	\begin{abstract}
	We show that certain extensions of classifiable $C^{*}$-algebra are strongly classified by the associated six-term exact sequence in $K$-theory together with the positive cone of $K_{0}$-groups of the ideal and quotient.  We apply our result to give a complete classification of graph $C^{*}$-algebras with exactly one ideal.  
	\end{abstract}

        \maketitle

\section{Introduction}

The classification program for $C^*$-algebras has for the most part
progressed independently for the classes of infinite and finite
$C^*$-algebras, and  great strides have been made in this program for each
of these classes.  In the finite case, Elliott's Theorem classifies
all AF-algebras up to stable isomorphism by the ordered $K_0$-group.
In the infinite case, there are a number of results for purely
infinite $C^*$-algebras.  The Kirchberg-Phillips Theorem classifies
certain simple purely infinite $C^*$-algebras up to stable isomorphism
by the $K_0$-group together with the $K_1$-group.  For nonsimple
purely infinite $C^*$-algebras many partial results have been
obtained: R\o rdam has shown that certain purely infinite
$C^*$-algebras with exactly one proper nontrivial ideal are
classified up to stable isomorphism by the associated six-term exact
sequence of $K$-groups \cite{extpurelyinf}, the second named author has shown that
nonsimple Cuntz-Krieger algebras satisfying Condition (II) are
classified up to stable isomorphism by their filtered $K$-theory
\cite[Theorem~4.2]{gr:ckalg}, and Meyer and Nest have shown that
certain purely infinite $C^*$-algebras with a linear ideal lattice are
classified up to stable isomorphism by their filtrated $K$-theory
\cite[Theorem~4.14]{rmrn:uctkkx}.  However, in all of these situations
the nonsimple $C^*$-algebras that are classified have the property
that they are either AF-algebras or purely infinite, and consequently
all of their ideals and quotients are of the same type.

Recently, the authors have provided
a framework for classifying nonsimple $C^*$-algebras that are not
necessarily AF-algebras or purely infinite $C^*$-algebras.  In
particular, the authors have shown in \cite{ERRshift} that
certain extensions of classifiable $C^*$-algebras may be classified up
to stable isomorphism by their associated six-term exact sequence in $K$-theory.  This has allowed for the classification of certain nonsimple $C^*$-algebras in which there are ideals and quotients of
mixed type (some finite and some infinite).  The results in \cite{ERRshift} was then used by the first named author and Tomforde in \cite{semt_classgraphalg} to classify a certain class of non-simple graph $C^{*}$-algebras, showing that graph $C^{*}$-algebras with exactly one non-trivial ideal can be classified up to stable isomorphism by their associated six-term exact sequence in $K$-theory.  The authors in \cite{err:fullext} then showed that all non-unital graph $C^{*}$-algebras with exactly one non-trivial ideal can be classified up to isomorphism by their associated six-term exact sequence in $K$-theory.  In this paper, we complete the classification of graph $C^{*}$-algebras with exactly one non-trivial ideal by classifying those that are unital.  Our methods here differ rather dramatically from the methods in \cite{semt_classgraphalg} and \cite{err:fullext}.  In particular, we use the traditional methods of classification via existence and uniqueness theorems.  As a consequence, for unital graph $C^{*}$-algebras $\mathfrak{A}$ and $\mathfrak{B}$ with exactly one non-trivial ideal, then any isomorphism between the associated six-term exact sequence in $K$-theory which preserves the unit lifts to an isomorphism from $\mathfrak{A}$ to $\mathfrak{B}$.

\section{Preliminaries}

\subsection{$C^{*}$-algebras over topological spaces} Let $X$ be a topological space and let $\mathbb{O}( X)$ be the set of open subsets of $X$, partially ordered by set inclusion $\subseteq$.  A subset $Y$ of $X$ is called \emph{locally closed} if $Y = U \setminus V$ where $U, V \in \mathbb{O} ( X )$ and $V \subseteq U$.  The set of all locally closed subsets of $X$ will be denoted by $\mathbb{LC}(X)$.  The set of all connected, non-empty, locally closed subsets of $X$ will be denoted by $\mathbb{LC}(X)^{*}$.  

The partially ordered set $( \mathbb{O} ( X ) , \subseteq )$ is a \emph{complete lattice}, that is, any subset $S$ of $\mathbb{O} (X)$ has both an infimum $\bigwedge S$ and a supremum $\bigvee S$.  More precisely, for any subset $S$ of $\mathbb{O} ( X )$, 
\begin{equation*}
\bigwedge_{ U \in S } U = \left( \bigcap_{ U \in S } U \right)^{\circ} \quad \mathrm{and} \quad \bigvee_{ U \in S } U = \bigcup_{ U \in S } U.
\end{equation*}

For a $C^{*}$-algebra $\mathfrak{A}$, let $\mathbb{I} ( \mathfrak{A} )$ be the set of closed ideals of $\mathfrak{A}$, partially ordered by $\subseteq$.  The partially ordered set $( \mathbb{I} ( \mathfrak{A} ), \subseteq )$ is a complete lattice.  More precisely, for any subset $S$ of $\mathbb{I} ( \mathfrak{A} )$, 
\begin{equation*}
\bigwedge_{ \mathfrak{I} \in S } \mathfrak{I} = \bigcap_{ \mathfrak{I} \in S } \mathfrak{I}  \quad \mathrm{and} \quad \bigvee_{ \mathfrak{I} \in S } \mathfrak{I} = \overline{ \sum_{ \mathfrak{I} \in S } \mathfrak{I} }.
\end{equation*}

\begin{defin}
Let $\mathfrak{A}$ be a $C^{*}$-algebra.  Let $\mathrm{Prim} ( \mathfrak{A} )$ denote the \emph{primitive ideal space} of $\mathfrak{A}$, equipped with the usual hull-kernel topology, also called the Jacobson topology.

Let $X$ be a topological space.  A \emph{$C^{*}$-algebra over $X$} is a pair $( \mathfrak{A} , \psi )$ consisting of a $C^{*}$-algebra $\mathfrak{A}$ and a continuous map $\ftn{ \psi }{ \mathrm{Prim} ( \mathfrak{A} ) }{ X }$.  A $C^{*}$-algebra over $X$, $( \mathfrak{A} , \psi )$, is \emph{separable} if $\mathfrak{A}$ is a separable $C^{*}$-algebra.  We say that $( \mathfrak{A} , \psi )$ is \emph{tight} if $\psi$ is a homeomorphism.  
\end{defin}

We always identify $\mathbb{O} ( \mathrm{Prim} ( \mathfrak{A} ) )$ and $\mathbb{I} ( \mathfrak{A} )$ using the lattice isomorphism
\begin{equation*}
U \mapsto \bigcap_{ \mathfrak{p} \in \mathrm{Prim} ( \mathfrak{A} ) \setminus U } \mathfrak{p}.
\end{equation*}
 Let $( \mathfrak{A} , \psi )$ be a $C^{*}$-algebra over $X$.  Then we get a map $\ftn{ \psi^{*} }{ \mathbb{O} ( X ) }{ \mathbb{O} ( \mathrm{Prim} ( \mathfrak{A} ) )  \cong \mathbb{I} ( \mathfrak{A} ) }$ defined by
\begin{equation*}
U \mapsto \setof{ \mathfrak{p} \in \mathrm{Prim} ( \mathfrak{A} ) }{ \psi ( \mathfrak{p} ) \in U } = \mathfrak{A}( U )
\end{equation*}
For $Y = U \setminus V \in \mathbb{LC} ( X )$, set $\mathfrak{A}( Y ) = \mathfrak{A} ( U ) / \mathfrak{A}(V)$.   By Lemma~2.15 of \cite{rmrn:bootstrap}, $\mathfrak{A} ( Y)$ does not depend on $U$ and $V$.

\begin{examp}
For any $C^{*}$-algebra $\mathfrak{A}$, the pair $( \mathfrak{A} , \id_{ \mathrm{Prim} ( \mathfrak{A} ) } )$ is a tight $C^{*}$-algebra over $\mathrm{Prim} ( \mathfrak{A} )$.  For each $U \in \mathbb{O} ( \mathrm{Prim} ( \mathfrak{A} ) )$, the ideal $\mathfrak{A} ( U )$ equals $\bigcap_{ \mathfrak{p} \in \mathrm{Prim} ( \mathfrak{A} ) \setminus U } \mathfrak{p}$. 
\end{examp}

\begin{examp}\label{Xn}
Let $X_{n} = \{ 1, 2, \dots, n \}$ partially ordered with $\leq$.  Equip $X_{n}$ with the Alexandrov topology, so the non-empty open subsets are 
\begin{equation*}
[a,n ] = \setof{ x \in X }{ a \leq x \leq n}
\end{equation*}
for all $a \in X_{n}$; the non-empty closed subsets are $[1,b]$ with $b \in X_{n}$, and the  non-empty locally closed subsets are those of the form $[a,b]$ with $a, b \in X_{n}$ and $a \leq b$.  Let $( \mathfrak{A} , \phi )$ be a $C^{*}$-algebra over $X_{n}$.  We will use the following notation throughout the paper:  
\begin{equation*}
\mathfrak{A} [k] = \mathfrak{A} ( \{ k \} ),\ \mathfrak{A}[a,b] = \mathfrak{A} ( [ a , b ] ),\ \text{and} \ \mathfrak{A}(i, j] = \mathfrak{A}[i+1, j].
\end{equation*}
Using the above notation we have ideals $\mathfrak{A} [ a, n ]$ such that 
\begin{equation*}
\{ 0 \} \unlhd \mathfrak{A} [ n]  \unlhd \mathfrak{A} [ n -1, n ]  \unlhd \cdots \unlhd \mathfrak{A}  [ 2, n ]  \unlhd \mathfrak{A}  [1,n] = \mathfrak{A}.
\end{equation*}  

\end{examp}

\begin{defin}
Let $\mathfrak{A}$ and $\mathfrak{B}$ be $C^{*}$-algebras over $X$.  A homomorphism $\ftn{ \phi }{ \mathfrak{A} }{ \mathfrak{B} }$ is \emph{$X$-equivariant} if $\phi ( \mathfrak{A} (U) ) \subseteq \mathfrak{B} ( U )$ for all $U \in \mathbb{O}(X)$.  Hence, for every $Y = U \setminus V$, $\phi$ induces a homomorphism $\ftn{ \phi_{Y} }{ \mathfrak{A} ( Y ) }{ \mathfrak{B} (Y) }$.  Let $\catc(X)$ be the category whose objects are $C^{*}$-algebras over $X$ and whose morphisms are $X$-equivariant homomorphisms.

An $X$-equivariant homomorphism $\ftn{ \phi }{ \mathfrak{A} }{ \mathfrak{B} }$ is said to be a \emph{full $X$-equivariant homomorphism} if for all $Y \in \mathbb{LC}(X)$, $\phi_{Y} ( a )$ is norm-full in $\mathfrak{B}(Y)$ for all norm-full elements $a \in \mathfrak{A}(Y)$, i.e., the closed ideal of $\mathfrak{B} (Y)$ generated by $\phi_{Y} (a)$ is $\mathfrak{B}(Y)$ whenever the closed ideal of $\mathfrak{A}(Y)$ generated by $a$ is $\mathfrak{A}(Y)$.  
\end{defin}

\begin{remar}\label{r:X2hom}
Suppose $\mathfrak{A}$ and $\mathfrak{B}$ are tight $C^{*}$-algebras over $X_{n}$.  Then it is clear that $\ftn{ \phi }{ \mathfrak{A} }{ \mathfrak{B} }$ is an isomorphism if and only if $\phi$ is a $X_{n}$-equivariant isomorphism.  

It is easy to see that if $\mathfrak{A}$ and $\mathfrak{B}$ are tight $C^{*}$-algebras over $X_{2}$, then $\ftn{ \phi }{ \mathfrak{A} }{ \mathfrak{B} }$ is a full $X_{2}$-equivariant homomorphism if and only if $\phi$ is an $X_{2}$-equivariant homomorphism and $\phi_{\{1\}}$ and $\phi_{\{2\}}$ are injective.  Also, if $\mathfrak{A}$ and $\mathfrak{A}[2]$ have non-zero projections $p$ and $q$ respectively, then there exists $\epsilon > 0$ such that if $\ftn{ \phi }{ \mathfrak{A} }{ \mathfrak{B} }$ is a full $X_{2}$-equivariant homomorphism and $\ftn{ \psi }{ \mathfrak{A} }{ \mathfrak{B} }$ is a homomorphism such that 
\begin{align*}
\| \phi (p) - \psi(p) \| < 1 \qquad \| \phi (q) - \psi(q) \| < 1,
\end{align*}
then $\psi$ is a full $X_{2}$-equivariant homomorphism.
\end{remar}

\begin{remar}
Let $\mathfrak{e}_{i} : 0 \to \mathfrak{B}_{i} \to  \mathfrak{E}_{i} \to \mathfrak{A}_{i} \to 0$ be an extension for $i = 1, 2$.  Note that $\mathfrak{E}_{i}$ can be considered as a $C^{*}$-algebra over $X_{2} = \{ 1,2 \}$ by sending $\emptyset$ to the zero ideal, $\{2\}$ to the image of $\mathfrak{B}_{i}$ in $\mathfrak{E}_{i}$, and $\{1 , 2 \}$ to $\mathfrak{E}_{i}$.  Hence, there exists a one-to-one correspondence between $X_{2}$-equivariant homomorphisms $\ftn{ \phi }{ \mathfrak{E}_{1} }{ \mathfrak{E}_{2} }$ and homomorphisms from $\mathfrak{e}_{1}$ and $\mathfrak{e}_{2}$.  
\end{remar}

\subsection{The ideal related $K$-theory of $\mathfrak{A}$}

\begin{defin}
Let $X$ be a topological space and let $\mathfrak{A}$ be a $C^{*}$-algebra over $X$.  For open subsets $U_{1} , U_{2} , U_{3}$ of $X$ with $U_{1} \subseteq U_{2} \subseteq U_{3}$, set $Y_{1} = U_{2} \setminus U_{1}, Y_{2} = U_{3} \setminus U_{1}, Y_{3} = U_{3} \setminus U_{1} \in \mathbb{LC} ( X )$.  Then the diagram
\begin{equation*}
\xymatrix{
K_{0} ( \mathfrak{A} ( Y_{1}  ) ) \ar[r]^{ \iota_{*} } & K_{0} ( \mathfrak{A} ( Y_{2} ) ) \ar[r]^{ \pi_{*} } & K_{0} ( \mathfrak{A} ( Y_{3} ) ) \ar[d]^{\partial_{*}} \\
K_{1} ( \mathfrak{A} ( Y_{3}  ) ) \ar[u]^{ \partial_{*}} & K_{1} ( \mathfrak{A} ( Y_{2} ) ) \ar[l]^{ \pi_{*} } & K_{1} ( \mathfrak{A} ( Y_{1} ) ) \ar[l]^{\iota_{*}}
}
\end{equation*}
is an exact sequence.  The \emph{ideal related $K$-theory of $\mathfrak{A}$}, $K_{X} ( \mathfrak{A} )$, is the collection of all $K$-groups thus occurring and the natural transformations $\{ \iota_{*}, \pi_{*}, \partial_{*} \}$.  The \emph{ideal related, ordered $K$-theory of $\mathfrak{A}$}, $K_{X}^{+} ( \mathfrak{A} )$, is $K_{X}( \mathfrak{A} )$ of $\mathfrak{A}$ together with $K_{0} ( \mathfrak{A} ( Y ) )_{+}$ for all $Y \in \mathbb{LC} ( X )$.

Let $\mathfrak{A}$ and $\mathfrak{B}$ be $C^{*}$-algebras over $X$, we will say that $\ftn{\alpha }{ K_{X} ( \mathfrak{A} ) }{  K_{X} ( \mathfrak{B} ) }$ is an \emph{isomorphism} if for all $Y  \in \mathbb{LC} ( X )$, there exists a graded group isomorphism
\begin{equation*}
\ftn{\alpha_{Y, *}} { K_{*} ( \mathfrak{A} (Y) ) }{ K_{*} ( \mathfrak{B} (Y) ) }
\end{equation*}
preserving all natural transformations.  We say that $\ftn{ \alpha }{  K_{X}^{+} ( \mathfrak{A} ) }{  K_{X}^{+} ( \mathfrak{B} ) }$ is an \emph{isomorphism} if there exists an isomorphism $\ftn{\alpha }{ K_{X} ( \mathfrak{A} ) }{  K_{X} ( \mathfrak{B} ) }$  in such a way that $\alpha_{Y, 0}$ is an order isomorphism for all $Y \in \mathbb{LC}(X)$. 
\end{defin}

\begin{remar}\label{filtrated}
Meyer-Nest in \cite{rmrn:uctkkx} defined a similar functor $\mathrm{FK}_{X} ( - )$ which they called filtrated $K$-theory.  For all known cases in which there exists a UCT, the natural transformation from $\mathrm{FK}_{X}( - )$ to $K_{X} ( - )$ is an equivalence.  In particular, this is true for the space $X_{n}$.
\end{remar}

If $Y \in \mathbb{LC}(X)$ such that $Y = Y_{1} \sqcup Y_{2}$ with two disjoint relatively open subsets $Y_{1} , Y_{2} \in \mathbb{O} ( Y ) \subseteq \mathbb{LC} (X)$, then $\mathfrak{A} (Y) \cong \mathfrak{A} ( Y_{1} ) \oplus \mathfrak{A} ( Y_{2} )$ for any $C^{*}$-algebra over $X$.  Moreover, there is a natural isomorphism $K_{*} ( \mathfrak{A} (Y) )$ to $K_{*} ( \mathfrak{A}( Y_{1} ) ) \oplus K_{*} ( \mathfrak{A} ( Y_{2} ) )$ which is a positive isomorphism from $K_{0} ( \mathfrak{A} (Y) )$ to $K_{0} ( \mathfrak{A}( Y_{1} ) ) \oplus K_{0} ( \mathfrak{A} ( Y_{2} ) )$.  If $X$ is finite, then any locally closed subset is a disjoint union of its connected components.  Therefore, we lose no information when we replace $\mathbb{LC} ( X )$ by the subset $\mathbb{LC} ( X )^{*}$.  

\begin{notat}
Let $\mathcal{N}$ be the bootstrap category of Rosenberg and Schochet in \cite{uct}.

Let $\mathfrak{K}\mathfrak{K}(X)$ be the category whose objects are separable $C^{*}$-algebras over $X$ and the set of morphisms is $\kk( X ; \mathfrak{A} , \mathfrak{B} )$.  For a finite topological space $X$, let $\mathcal{B} (X) \subseteq \mathfrak{K}\mathfrak{K} (X)$ be the bootstrap category of Meyer and Nest in \cite{rmrn:bootstrap}.  By Corollary~4.13 of \cite{rmrn:bootstrap}, if $\mathfrak{A}$ is a nuclear $C^{*}$-algebra over $X$, then $\mathfrak{A} \in \mathcal{B} (X)$ if and only if $\mathfrak{A} ( \{ x \} ) \in \mathcal{N}$ for all $x \in X$. 
\end{notat}

\begin{theor}(Bonkat \cite{bonkat} and Meyer-Nest \cite{rmrn:uctkkx})\label{t:uctX}
Let $\mathfrak{A}$ and $\mathfrak{B}$ be in $\mathfrak{K}\mathfrak{K}(X_{n})$ such that $\mathfrak{A}$ is in $\mathcal{B} ( X_{n} )$, then the sequence
\begin{align*}
0 \to \Ext_{\mathcal{NT} }^{1} ( \mathsf{FK}_{X_{n}} ( \mathfrak{A} )[1], \mathsf{FK}_{X_{n}} ( \mathfrak{B} ) ) \overset{ \delta }{ \to } \kk (X_{n}; \mathfrak{A} , \mathfrak{B} ) \overset{ \Gamma }{ \to } \Hom_{ \mathcal{NT} } ( \mathrm{FK}_{X_{n}} ( \mathfrak{A} ) , \mathrm{FK}_{X_{n}} ( \mathfrak{B} )  \to 0
\end{align*}
is exact.  Consequently, if $\mathfrak{B}$ is in $\mathcal{B} ( X_{n} )$, then an isomorphism from $\mathrm{FK}_{X_{n}} ( \mathfrak{A} )$ to $\mathrm{FK}_{X_{n}} ( \mathfrak{B} )$ lifts to an invertible element in $\kk (X_{n}; \mathfrak{A} , \mathfrak{B} )$.
\end{theor}

\begin{corol}\label{c:uctX}
Let $\mathfrak{A}$ and $\mathfrak{B}$ be in $\mathcal{B} ( X_{n} )$.  Then an isomorphism from $K_{X_{n}} ( \mathfrak{A} )$ to $K_{X_{n}} ( \mathfrak{B} )$ lifts to an invertible element in $\kk (X_{n}; \mathfrak{A} , \mathfrak{B} )$.
\end{corol}

\begin{proof}
This follows from Remark~\ref{filtrated} and Theorem~\ref{t:uctX}.
\end{proof}

\begin{remar}
Let $x \in \kk ( X_{n} ; \mathfrak{A} , \mathfrak{B} )$ be an invertible element.  Then $K_{X_{n}} ( x )$ will denote the isomorphism from $K_{X_{n}} ( \mathfrak{A} )$ to $K_{X_{n}} ( \mathfrak{B} )$ given by $\Gamma ( x )$ where we have identified $K_{X_{n}} ( \mathfrak{A} )$ with $\mathrm{FK}_{X_{n}} ( \mathfrak{A} )$ and $K_{X_{n}} ( \mathfrak{B} )$ with $\mathrm{FK}_{X_{n}} ( \mathfrak{B} )$.
\end{remar}

\subsection{Functors}

We now define some functors that will be used throughout the rest of the paper.  Let $X$ and $Y$ be topological spaces.  For every continuous function $\ftn{f}{X}{Y}$ we have a functor
\begin{equation*}
\ftn{ f }{\catc(X) } { \catc( Y) }, \quad ( A , \psi ) \mapsto (A , f \circ \psi ). 
\end{equation*}  

\begin{itemize}
\item[(1)] Define $\ftn{ g_{ X }^{1} }{ X }{ X_{1} }$ by $g_{X}^{1} ( x )  = 1$.  Then $g_{X}^{1}$ is continuous.  Note that the induced functor $\ftn{ g_{X}^{1} }{ \catc(X) } { \catc( X_{1} ) }$ is the forgetful functor.

\item[(2)]  Let $U$ be an open subset of $X$.  Define $\ftn{ g_{U, X}^{ 2 }}{ X } { X_{2} }$ by $g_{U,X}^{2}(x) = 1$ if $x \notin U$ and $g_{U, X }^{ 2}(x) = 2$ if $x \in U$.  Then $g_{U,X }^{2}$ is continuous.  Thus the induced functor
\begin{equation*}
\ftn{ g_{ U, X}^{2} }{ \catc(X) } { \catc( X_{2} ) }
\end{equation*} 
is just specifying the extension $0 \to \mathfrak{A}(U) \to \mathfrak{A} \to \mathfrak{A} / \mathfrak{A}(U) \to 0$.

\item[(3)] We can generalize (2) to finitely many ideals.  Let $U_{1} \subseteq U_{2} \subseteq \cdots \subseteq U_{n} = X$ be open subsets of $X$.  Define $\ftn{g_{ U_{1} , U_{2} , \dots, U_{n} , X }^{n} }{ X }{ X_{n} }$ by $ g_{ U_{1} , U_{2} , \dots, U_{n} , X }^{n}  (x) = n-k+1$ if $x \in U_{k} \setminus U_{k-1}$.  Then $g_{ U_{1} , U_{2} , \dots, U_{n} , X }^{n}$ is continuous.  Therefore, any $C^{*}$-algebra with ideals $0 \unlhd \mathfrak{I}_{1} \unlhd \mathfrak{I}_{2} \unlhd \cdots \mathfrak{I}_{n} = \mathfrak{A}$ can be made into a $C^{*}$-algebra over $X_{n}$.

\item[(4)]  For all $Y \in \mathbb{LC} ( X )$, $\ftn{ r_{X}^{Y} }{ \catc(X) }{ \catc(Y) }$ is the restriction functor defined in Definition~2.19 of \cite{rmrn:bootstrap}

\item[(5)] If $\ftn{ f }{ X }{ Y }$ is an embedding of a subset with the subspace topology, we write
\begin{align*}
\ftn{i_{X}^{Y} = f_{*} }{ \catc(X) }{ \catc(Y) }.
\end{align*}
\end{itemize}
By Proposition~3.4 of \cite{rmrn:bootstrap}, the functors defined above induce functors from $\kkc( X )$ to $\kkc ( Z)$, where $Z = Y, X_{1}, X_{n}$.
 
 \subsection{Graph $C^{*}$-algebras}
 
 A \emph{graph}\index{graph} $(E^0, E^1, r, s)$ consists of a countable set $E^0$ of vertices, a countable set $E^1$ of edges, and maps $r : E^1 \to E^0$ and $s : E^1 \to E^0$ identifying the range and source of each edge.  If $E$ is a graph, the \emph{graph $C^*$-algebra} $C^*(E)$ is the universal $C^{*}$-algebra generated by mutually orthogonal projections $\{ p_v
: v \in E^0 \}$ and partial isometries $\{ s_e : e \in E^1 \}$ with mutually orthogonal
ranges satisfying
\begin{enumerate}
\item $s_e^* s_e = p_{r(e)}$ \quad  for all $e \in E^1$
\item $s_es_e^* \leq p_{s(e)}$ \quad for all $e \in E^1$
\item $p_v = \sum_{\{ e \in E^1 : s(e) = v \}} s_es_e^* $ \quad for all $v$ with $0 < | s^{-1}(v) | < \infty$.
\end{enumerate}

\section{Meta-theorems}\label{s:meta}
In many cases one can obtain a classification result for a class of unital $C^{*}$-algebras $\mathcal{C}$ by obtaining a classification result for the class $\mathcal{C} \otimes \K$, where each object in $\mathcal{C} \otimes \K$ is the stabilization of an object in $\mathcal{C}$.  A meta-theorem of this sort was proved by the first and second named authors in \cite[Theorem~11]{segr:extpurely}.  It was shown there that if $\mathcal{C}$ is a subcategory of the category of $C^{*}$-algebras, $\catc$, and if $F$ is a functor from $\mathcal{C}$ to an abelian category such that an isomorphism $F( \mathfrak{A} \otimes \K ) \cong F( \mathfrak{B} \otimes \K )$ lifts to an isomorphism in $\mathfrak{A} \otimes \K \cong \mathfrak{B} \otimes \K$, then under suitable conditions, we have that $F( \mathfrak{A} ) \cong F( \mathfrak{B} )$ implies $\mathfrak{A} \cong \mathfrak{B}$.  In \cite{gr:ckalg}, the second and third named authors improved this result by showing that the isomorphism $F( \mathfrak{A} ) \cong F( \mathfrak{B} )$ lifts to an isomorphism from $\mathfrak{A}$ to $\mathfrak{B}$.  

In this section, we improve these results in order to deal with cases when $\mathcal{C}$ is a category (not necessarily a subcategory of $\catc$) and there exists a functor from $\mathcal{C}$ to $\catc$.  An example of such a category is the category of $C^{*}$-algebras over $\{1,2\}$, where $\{ 1 , 2 \}$ is given the discrete topology.  Then $\mathcal{C}$ is not a subcategory of $\catc$ but the forgetful functor (forgetting the $\{1,2\}$-structure) is a functor from $\mathcal{C}$ to $\catc$.  We also replace the condition of proper pure infiniteness by the stable weak cancellation property. 

\begin{defin}
A $C^{*}$-algebra $\mathfrak{A}$ is said to have the \emph{weak cancellation property} if $p$ is Murray-von Neumann equivalent to $q$ whenever $p$ and $q$ generate the same ideal $\mathfrak{I}$ and $[ p ] = [q]$ in $K_{0} ( \mathfrak{I} )$.  A $C^{*}$-algebra is said to have the \emph{stable weak cancellation property} if $\mathsf{M}_{n} ( \mathfrak{A} )$ has the weak cancellation property for all $n \in \N$.
\end{defin}

\begin{theor}(cf.~\cite[Theorem~11]{segr:extpurely})\label{t:metaiso}
Let $\mathcal{C}$ and $\mathcal{D}$ be categories, let $\catc$ be the category of $C^{*}$-algebras, and let $\ab$ be the category of abelian groups.  Suppose we have covariant functors $\ftn{ F }{ \mathcal{C} }{ \catc }$, $\ftn{ G }{ \mathcal{C} }{ \mathcal{D} }$, and $\ftn{H}{ \mathcal{D} }{ \ab}$ such that 
\begin{enumerate}
\item $H \circ G = K_{0} \circ F$.

\item For objects $\mathfrak{A}$ in $\mathcal{C}$, there exist an object $\mathfrak{A}_{\K}$ and a morphism $\ftn{ \kappa_{ \mathfrak{A} } }{ \mathfrak{A} }{ \mathfrak{A}_{\K} }$ such that $G( \kappa_{ \mathfrak{A} } )$ is an isomorphism in $\mathcal{D}$, $F( \mathfrak{A}_{\K} ) = F( \mathfrak{A} ) \otimes \K$, and $F( \kappa_{ \mathfrak{A} } ) = \id_{ F( \mathfrak{A} ) } \otimes e_{11}$. 

\item For all objects $\mathfrak{A}$ and $\mathfrak{B}$ in $\mathcal{C}$, every isomorphism $G( \mathfrak{A}_{\K} )$ to $G( \mathfrak{B}_{\K} )$ is induced by an isomorphism from $\mathfrak{A}_{\K}$ to $\mathfrak{B}_{\K}$.

\end{enumerate}
Let $\mathfrak{A}$ and $\mathfrak{B}$ be given such that $F(\mathfrak{A})$ and $F( \mathfrak{B} )$ are unital $C^{*}$-algebras.  Let $\ftn{ \rho }{ G( \mathfrak{A} )  }{ G( \mathfrak{B} )}$ be an isomorphism such that $H( \rho ) ( [ 1_{F(\mathfrak{A}) } ] ) = [ 1_{ F( \mathfrak{B} ) } ]$.  If $F(\mathfrak{B})$ has the stable weak cancellation property, then $F( \mathfrak{A} ) \cong F( \mathfrak{B} )$.
\end{theor}

\begin{proof}
Note that $G( \kappa_{ \mathfrak{A} } )$ and $G( \kappa_{ \mathfrak{B} } )$ are isomorphisms.  Therefore $G( \kappa_{ \mathfrak{B} } ) \circ \rho \circ G( \kappa_{ \mathfrak{A} } )^{-1}$ is an isomorphism from $G( \mathfrak{A}_{ \K } )$ to $G( \mathfrak{B}_{\K} )$.  Thus, there exists an isomorphism $\ftn{ \phi }{ \mathfrak{A}_{\K} }{ \mathfrak{B}_{\K} }$ such that $G( \phi ) = G( \kappa_{ \mathfrak{B} } ) \circ \rho \circ G( \kappa_{ \mathfrak{A} } )^{-1}$.  

Set $\psi = F( \phi )$.  Then $\ftn{ \psi }{ F( \mathfrak{A} )\otimes \K }{ F( \mathfrak{B} ) \otimes \K }$ is a $*$-isomorphism such that
\begin{align*}
K_{0} ( \psi ) &= K_{0} ( F ( \phi ) ) = H( G( \kappa_{ \mathfrak{B} } ) \circ \rho \circ G( \kappa_{ \mathfrak{A} } )^{-1} )  =  H( G( \kappa_{ \mathfrak{B} }) ) \circ H( \rho ) \circ H( G( \kappa_{ \mathfrak{A} } )^{-1} ) \\
&= K_{0} ( F( \kappa_{\mathfrak{B} } ) ) \circ H( \rho )  \circ K_{0} ( F( \kappa_{\mathfrak{A} } ) )^{-1} = K_{0} ( \id_{ F( \mathfrak{B} )} \otimes e_{11} ) \circ H( \rho ) \circ K_{0} ( \id_{ F( \mathfrak{A} )} \otimes e_{11} )^{-1}.
\end{align*}
Hence, 
\begin{align*}
K_{0} ( \psi )( [ 1_{F( \mathfrak{A} ) } \otimes e_{11} ] ) &= K_{0} ( \id_{ F( \mathfrak{B} )} \otimes e_{11} )  \circ H( \rho ) \circ K_{0} ( \id_{ F( \mathfrak{A} )} \otimes e_{11} )^{-1} (  [ 1_{F( \mathfrak{A} ) } \otimes e_{11} ] )  \\
	&=  K_{0} ( \id_{ F( \mathfrak{B} )} \otimes e_{11} ) \circ H( \rho ) ( [ 1_{ F( \mathfrak{A} ) } ] ) \\
	&= K_{0} ( \id_{ F( \mathfrak{B} ) } \otimes e_{11} )( [ 1_{ F( \mathfrak{B} ) } ] ) \\
	&= [1_{ F( \mathfrak{B} ) } \otimes e_{11} ].
\end{align*}

Stable weak cancellation implies that there exists $v \in F( \mathfrak{B} ) \otimes \K$ such that $v^{*} v = \psi ( 1_{ F( \mathfrak{A} ) \otimes e_{11} } )$ and $vv^{*} = 1_{ F( \mathfrak{B} ) } \otimes e_{11}$ since $\psi ( 1_{ F( \mathfrak{A} ) } \otimes e_{11} )$ and $1_{ F( \mathfrak{B} ) } \otimes e_{11}$ are full projections in $F( \mathfrak{B} )\otimes \K$.  Set $\gamma (x) = v \psi ( x \otimes e_{11} ) v^{*}$.  Arguing as in the proof of \cite[Theorem~11]{segr:extpurely}, $\gamma$ is an isomorphism from $F( \mathfrak{A} ) \otimes e_{11}$ to $F( \mathfrak{B} ) \otimes e_{11}$.  Hence, $F( \mathfrak{A} ) \cong F( \mathfrak{B} )$.  
\end{proof}

\begin{theor}(cf.~\cite[Theorem~2.1]{rr_cexpure}) \label{t:metathm}
Let $\mathcal{C}$ be a subcategory of $\catc(X)$.  Moreover, $\mathcal{C}$ is assumed to be closed under tensoring by $\mathsf{M}_{2} ( \C )$ and $\K$ and contains the canonical embeddings $\ftn{ \kappa_{1} }{ \mathfrak{A} }{ \mathsf{M}_{2} ( \mathfrak{A} ) }$ and $\ftn{ \kappa }{ \mathfrak{A} }{ \mathfrak{A} \otimes \K }$ as morphisms for every object $\mathfrak{A}$ in $\mathcal{C}$.  Assume there is a functor $\ftn{\mathsf{F}}{ \mathcal{C} }{ \mathcal{D} }$ satisfying 
\begin{itemize}
\item[(1)] For $\mathfrak{A}$ in $\mathcal{C}$, the embeddings $\ftn{ \kappa_{1} }{ \mathfrak{A} }{ \mathsf{M}_{2} ( \mathfrak{A} ) }$ and $\ftn{ \kappa }{ \mathfrak{A} }{ \mathfrak{A} \otimes \K }$ induce isomorphisms $\mathsf{F} ( \kappa_{1} )$ and $\mathsf{F} ( \kappa )$.

\item[(2)] For all objects $\mathfrak{A}$ and $\mathfrak{B}$ in $\mathcal{C}$ that are stable $C^{*}$-algebras, every isomorphism from $\mathsf{F} ( \mathfrak{A} )$ to $\mathsf{F} ( \mathfrak{B} )$ is induced by an isomorphism from $\mathfrak{A}$ to $\mathfrak{B}$.

\item[(3)] There exists a functor $\mathsf{G}$ from $\mathcal{D}$ to $\ab$ such that $\mathsf{G} \circ \mathsf{F} = K_{0}$.
\end{itemize}
Assume that every $X$-equivariant isomorphism between objects in $\mathcal{C}$ is a morphism in $\mathcal{C}$ and that for objects $\mathfrak{A}$ in $\mathcal{C}$, $\mathsf{F} ( \mathrm{Ad} ( u ) \vert_{\mathfrak{A} } ) = \id_{ \mathsf{F} (\mathfrak{A} ) }$ for every unitary $u \in \multialg{ \mathfrak{A} }$.  If $\mathfrak{A}$ and $\mathfrak{B}$ are objects $\mathcal{C}$ that are unital $C^{*}$-algebras such that $\mathfrak{A}$ and $\mathfrak{B}$ have the stable weak cancellation property and there is an isomorphism $\ftn{ \alpha }{ \mathsf{F} ( \mathfrak{A} ) }{ \mathsf{F} ( \mathfrak{B} ) }$ such that $\mathsf{G} ( \alpha ) ( [ 1_{ \mathfrak{A} } ] ) = [ 1_{ \mathfrak{B} } ]$, then there exists an isomorphism $\ftn{ \phi }{ \mathfrak{A} }{ \mathfrak{B} }$ in $\mathcal{C}$ such that $\mathsf{F} ( \phi ) = \alpha$.
\end{theor}

\begin{proof}
The difference between the statement of Theorem~2.1 of \cite{rr_cexpure} and statement of the theorem are 
\begin{itemize}
\item[(i)] $\mathcal{C}$ is assumed to be a subcategory of $\catc(X)$ instead of a subcategory of $\catc$. 

\item[(ii)] $\mathfrak{A}$ and $\mathfrak{B}$ are assumed to have the stable weak cancellation property instead of being properly infinite.
\end{itemize}
In the proof of Theorem~2.1 of \cite{rr_cexpure}, properly infinite was needed to insure that $\psi ( 1_{ \mathfrak{A} } \otimes e_{11} )$ is Murray-von Neumann equivalent to $1_{ \mathfrak{B} } \otimes e_{11}$, where $\ftn{ \psi }{ \mathfrak{A} \otimes \K }{ \mathfrak{B} \otimes \K }$ is the isomorphism from (2) that lifts the isomorphism from $\mathsf{F} ( \mathfrak{A} )$ to $\mathsf{F} ( \mathfrak{B} )$ that is induced by $\alpha$.  As in the proof of Theorem~\ref{t:metaiso}, we get that $\psi ( 1_{ \mathfrak{A} } \otimes e_{11} )$ is Murray-von Neumann equivalent to $1_{ \mathfrak{B} } \otimes e_{11}$.  Arguing as in the proof of Theorem~2.1 of \cite{rr_cexpure}, we get the desired result.

\end{proof}

\section{Classification results}
In this section, we show that $K_{X_{2}}^{+} ( - )$ is a strong classification functor for a class of $C^{*}$-algebras with exactly one proper nontrivial ideal containing $C^{*}$-algebras associated to finite graphs.  The results of this section will be used in the next section to show that $K_{X_{2}}^{+} ( - )$ together with the appropriate scale is a complete isomorphism invariant for $C^{*}$-algebras associated to graphs.  Moreover, in a forthcoming paper, we use these results to solve the following \emph{extension problem}: If $\mathfrak{A}$ fits into the following exact sequence
\begin{align*}
0 \to C^{*} ( E ) \otimes \K \to \mathfrak{A} \to C^{*} ( G) \to 0,
\end{align*}
where $C^{*} (E)$ and $C^{*} (G)$ are simple $C^{*}$-algebras, then when is $\mathfrak{A} \cong C^{*} (F)$ for some graph $F$?

\begin{theor}(Existence Theorem)\label{t:exist}
Let $\mathfrak{A}_{1}$ and $\mathfrak{A}_{2}$ be in $\mathcal{B}( X_{2})$ and let $x \in \kk( X_{2}; \mathfrak{A}_{1} , \mathfrak{A}_{2} )$ be an invertible element such that $\Gamma(x)_{Y}$ is a positive isomorphism for all $Y \in \mathbb{LC}(X_{2})$.  Suppose $0 \to \mathfrak{A}_{i}[2]   \to \mathfrak{A}_{i}  \to \mathfrak{A}_{i} [1]  \to 0$ is a full extension, $\mathfrak{A}_{i} [2]$ is a stable $C^{*}$-algebra, $\mathfrak{A}_{i}$ is a nuclear $C^{*}$-algebra with real rank zero, and either
\begin{itemize}
\item[(i)] $\mathfrak{A}_{i} [ 2 ]$ is a purely infinite simple $C^{*}$-algebra and $\mathfrak{A}_{i} [ 1]$ is an AF-algebra; or

\item[(ii)] $\mathfrak{A}_{i} [ 2 ]$ is an AF-algebra and $\mathfrak{A}_{i} [ 1 ]$ is a purely infinite simple $C^{*}$-algebra.
\end{itemize}
Then there exists an $X_{2}$-equivariant homomorphism $\ftn{ \phi }{ \mathfrak{A}_{1} \otimes \K }{ \mathfrak{A}_{2} \otimes \K }$ such that $\kk (X_{2};  \phi ) = \kk( X_{2} ; \id_{ \mathfrak{A}_{1} } \otimes e_{11} )^{-1} \times x \times \kk( X_{n} ; \id_{ \mathfrak{A}_{2} } \otimes e_{11} )$, and $\phi_{\{ 2\} }$ and $\phi_{\{1\}}$ are injective, where $\{ e_{ij} \}$ is a system of matrix units for $\K$.
\end{theor}

\begin{proof}
Set $y = \kk( X ; \id_{ \mathfrak{A}_{1} } \otimes e_{11} )^{-1} \times x \times \kk( X ; \id_{ \mathfrak{A}_{2} } \otimes e_{11} )$.  Note that by Lemma~3.10 and Theorem~3.8 of \cite{segrer:ccfis}, $\mathfrak{A}_{i} [ 2 ] \otimes \K$ satisfies the corona factorization property (see \cite{dkpn:cfp} for the definition of the corona factorization property).  Since $\mathfrak{A}_{i} [ k ]$ is an AF-algebra or an Kirchberg algebra, $\mathfrak{A}_{i}[k]$ has the stable weak cancellation.  By Lemma~3.15 of \cite{err:fullext}, $\mathfrak{A}_{i}$ has stable weak cancellation.  Let $\mathfrak{e}_{i}$ be the extension
\begin{align*}
0 \to \mathfrak{A}_{i}[2] \otimes \K  \to \mathfrak{A}_{i} \otimes \K  \to \mathfrak{A}_{i} [1] \otimes \K \to 0.
\end{align*}  
By Corollary~3.24 of \cite{err:fullext}, $\mathfrak{e}_{i}$ is a full extension since $\mathfrak{A}_{i}[1]$ has cancellation of projections (in the AF case) and $\mathfrak{A}_{i}[1]$ is properly infinite (in the purely infinite case).

\medskip

\emph{Case (i): $\mathfrak{A}_{i}[2]$ is a purely infinite simple $C^{*}$-algebra and $\mathfrak{A}_{i} [1]$ is an AF-algebra.}  By Theorem~3.3 of \cite{segrer:ccfis}, $r_{X_{2}}^{\{1\}}( y ) \times [ \tau_{ \mathfrak{e}_{2} } ] = [ \tau_{ \mathfrak{e}_{1} } ] \times r_{X_{2}}^{\{2\}} (y)$ in $\kk^{1} ( \mathfrak{A}_{1} [ 1] \otimes \K , \mathfrak{A}_{2} [2] \otimes \K )$.  Since $y$ is invertible in $\kk( X_{2} , \mathfrak{A}_{1} \otimes \K , \mathfrak{A}_{2} \otimes \K )$, we have that $r_{X_{2}}^{\{1\}}( y )$ is invertible in $\kk( \mathfrak{A}_{1} [1] \otimes \K , \mathfrak{A}_{2}[1] \otimes\K )$ and $\Gamma( r_{X_{2}}^{\{1\}}( y ) ) = \Gamma ( x )_{\{1\}}$ is a positive isomorphism.  Thus, by Elliott's classification \cite{af}, there exists an isomorphisms $\ftn{ \psi_{1} }{ \mathfrak{A}_{1} [ 1 ] \otimes \K}{ \mathfrak{A}_{2} [ 1 ] \otimes \K}$ such that $\kk ( \psi_{1} ) = r_{X_{2}}^{\{1\}}( y )$.  Since $y$ is invertible in $\kk( X_{2} , \mathfrak{A}_{1} \otimes \K , \mathfrak{A}_{2} \otimes \K )$, we have that $r_{X_{2}}^{\{2\}}( y )$ is invertible in $\kk( \mathfrak{A}_{1} [2] \otimes \K , \mathfrak{A}_{2}[2] \otimes\K )$.  Thus, by Kirchberg-Phillips classification (see \cite{kirchpure} and \cite{phillipspureinf}), there exists an isomorphism $\ftn{ \psi_{0} }{ \mathfrak{A}_{1} [ 2 ] \otimes \K }{ \mathfrak{A}_{2} [2] \otimes \K }$ such that $\kk( \psi_{0} ) = r_{X_{2}}^{\{2\}}( y )$.  By Lemma~4.5 of \cite{segrer:ccfis} and its proof, there exists a unitary $u \in \multialg{ \mathfrak{A}_{2}[2] \otimes \K }$ such that $\psi = ( \mathrm{Ad}(u) \circ \psi_{0} , \mathrm{Ad} (u) \circ \widetilde{\psi}_{0} , \psi_{1} )$ is an $X_{2}$-equivariant isomorphism from $\mathfrak{A}_{1} \otimes \K$ to $\mathfrak{A}_{2} \otimes \K$, where $\ftn{ \widetilde{\psi}_{0} }{ \multialg{ \mathfrak{A}_{1} [ 2 ] \otimes \K } }{  \multialg{ \mathfrak{A}_{1} [ 2 ] \otimes \K } }$ is the unique isomorphism extending $\psi_{0}$.  Note that $\kk ( \psi_{\{ k \} } ) = r_{X_{2}}^{\{k\}} ( y)$ for $k = 1, 2$.

Note that
\begin{align*}
0 \to i^{ X_{2} }_{\{2\} } ( ( \mathfrak{A}_{i} \otimes \K ) [ 2 ] ) \overset{ \lambda_{i}} { \to } \mathfrak{A}_{i} \otimes \K \overset{ \beta_{i} }{ \to }  i^{ X_{2} }_{ \{ 1 \} } ( (\mathfrak{A}_{i} \otimes \K)[1] ) \to 0
\end{align*}
is a semi-split extension of $C^{*}$-algebras over $X_{2}$ (see Definition~3.5 of \cite{rmrn:bootstrap}).  Set 
\begin{align*}
\mathfrak{I}_{i} = i^{ X_{2} }_{\{2\} } ( ( \mathfrak{A}_{i} \otimes \K ) [ 2 ] ) \qquad \text{and} \qquad \mathfrak{B}_{i} = i^{ X_{2} }_{ \{ 1 \} } ( (\mathfrak{A}_{i} \otimes \K)[1] ).
\end{align*}
By Theorem~3.6 of \cite{rmrn:bootstrap} (see also Korollar~3.4.6 of \cite{bonkat}), 
 \begin{align*}
\xymatrix{
 \kk( X_{2} ; \mathfrak{A}_{1} \otimes \K , \mathfrak{I}_{2}  ) \ar[r]^-{ (\lambda_{2} )_{*} } & \kk ( X_{2} ; \mathfrak{A}_{1} \otimes \K , \mathfrak{A}_{2} \otimes \K ) \ar[r]^-{ ( \beta_{2} )_{*} } &  \kk ( X_{2} ; \mathfrak{A}_{1}\otimes \K, \mathfrak{B}_{2} )
 }
 \end{align*}
is exact.  By Proposition~3.12 of \cite{rmrn:bootstrap}, $ \kk ( X_{2} ; \mathfrak{A}_{1} \otimes \K, \mathfrak{B}_{2} )$ and $\kk ( \mathfrak{A}_{1}[1] \otimes \K, \mathfrak{A}_{2} [1]  \otimes \K)$ are naturally isomorphic.  Hence, there exists $z \in \kk ( X_{2} ; \mathfrak{A}_{1} \otimes \K , \mathfrak{I}_{2} )$ such that $y - \kk( X_{2} ; \psi ) = z \times \kk ( X_{2} ; \lambda_{2} )$ since $\kk( \psi_{ \{1\}} ) = r_{X_{2}}^{\{1\}}(y)$.  

By Proposition~3.13 of \cite{rmrn:bootstrap}, $\kk ( X_{2} ; \mathfrak{A}_{1} \otimes \K , \mathfrak{I}_{2} )$ and $\kk ( \mathfrak{A}_{1} \otimes \K ,  (\mathfrak{A}_{2}  \otimes \K )[2] )$ are isomorphic.  By Theorem~8.3.3 of \cite{classentropy} (see also Hauptsatz 4.2 of \cite{kirchpure}), there exists a $*$-homomorphism $\ftn{ \eta }{ \mathfrak{A}_{1} \otimes \K }{ (\mathfrak{A}_{2} \otimes \K )[2] }$ such that $\kk ( \eta ) = \overline{z}$, where $\overline{z}$ is the image of $z$ under the isomorphism $\kk ( X_{2} ; \mathfrak{A}_{1} \otimes \K , \mathfrak{I}_{2} ) \cong \kk ( \mathfrak{A}_{1} \otimes \K ,  (\mathfrak{A}_{2}  \otimes \K )[2] )$.  Note that $\eta$ induces an $X_{2}$-equivariant homomorphism $\ftn{ \eta }{ \mathfrak{A}_{1} \otimes \K }{ \mathfrak{I}_{2} }$ such that $\kk( X_{2} ; \eta ) = z$.  

Set $\phi = \psi + (\lambda_{2} \circ \eta)$, where the sum is the Cuntz sum in $\multialg{ \mathfrak{A}_{2} \otimes \K }$.  Then $\ftn{ \phi}{ \mathfrak{A}_{1} \otimes \K } { \mathfrak{A}_{2} \otimes \K }$ is an $X_{2}$-equivariant homomorphism such that $\kk( X_{2} ; \phi ) = y$.  Since $\psi_{\{2\}}$ and $\psi_{\{1\}}$ are injective homomorphisms, $\phi_{\{2\}}$ and $\phi_{\{1\}}$ are injective homomorphisms. 

\medskip

\emph{Case (ii): $\mathfrak{A}_{i}[2]$ is an AF-algebra and $\mathfrak{A}_{i} [1]$ is a purely infinite simple $C^{*}$-algebra.}  By Theorem~3.3 of \cite{segrer:ccfis}, $r_{X_{2}}^{\{1\}}( y ) \times [ \tau_{ \mathfrak{e}_{2} } ] = [ \tau_{ \mathfrak{e}_{1} } ] \times r_{X_{2}}^{\{2\}} (y)$ in $\kk^{1} ( \mathfrak{A}_{1} [ 1] \otimes \K , \mathfrak{A}_{2} [2] \otimes \K )$.  Since $y$ is invertible in $\kk( X_{2} , \mathfrak{A}_{1} \otimes \K , \mathfrak{A}_{2} \otimes \K )$, we have that $r_{X_{2}}^{\{2\}}( y )$ is invertible in $\kk( \mathfrak{A}_{1} [2] \otimes \K , \mathfrak{A}_{2}[2] \otimes\K )$ and $\Gamma( r_{X_{2}}^{\{2\}}( y ) ) = \Gamma ( x )_{\{2\}}$ is an order isomorphism.  Thus, by Elliott's classification \cite{af}, there exists an isomorphism $\ftn{ \psi_{0} }{ \mathfrak{A}_{1} [ 2 ] \otimes \K}{ \mathfrak{A}_{2} [ 2 ] \otimes \K}$ such that $\kk ( \psi_{0} ) = r_{X_{2}}^{\{2\}}( y )$.  Since $y$ is invertible in $\kk( X_{2} , \mathfrak{A}_{1} \otimes \K , \mathfrak{A}_{2} \otimes \K )$, we have that $r_{X_{2}}^{\{1\}}( y )$ is invertible in $\kk( \mathfrak{A}_{1} [1] \otimes \K , \mathfrak{A}_{2}[1] \otimes\K )$.  Thus, by Kirchberg-Phillips classification (see \cite{kirchpure} and \cite{phillipspureinf}), there exists an isomorphism $\ftn{ \psi_{1} }{ \mathfrak{A}_{1} [ 1 ] \otimes \K }{ \mathfrak{A}_{2} [1] \otimes \K }$ such that $\kk( \psi_{1} ) = r_{X_{2}}^{\{1\}}( y )$.  By Lemma~4.5 of \cite{segrer:ccfis} and its proof, there exists a unitary $u \in \multialg{ \mathfrak{A}_{2}[2] \otimes \K }$ such that $\psi = ( \mathrm{Ad}(u) \circ \psi_{0} , \mathrm{Ad} (u) \circ \widetilde{\psi}_{0} , \psi_{1} )$ is an $X_{2}$-equivariant isomorphism from $\mathfrak{A}_{1} \otimes \K$ to $\mathfrak{A}_{2} \otimes \K$, where $\ftn{ \widetilde{\psi}_{0} }{ \multialg{ \mathfrak{A}_{1} [ 2 ] \otimes \K } }{  \multialg{ \mathfrak{A}_{1} [ 2 ] \otimes \K } }$ is the unique isomorphism extending $\psi_{0}$.  Note that $\kk ( \psi_{\{ k \} } ) = r_{X_{2}}^{\{k\}} ( y)$ for $k = 1, 2$.

Note that
\begin{align*}
0 \to i^{ X_{2} }_{\{ 2\} } ( ( \mathfrak{A}_{i} \otimes \K ) [ 2] ) \overset{ \lambda_{i}} { \to } \mathfrak{A}_{i} \otimes \K \overset{ \beta_{i} }{ \to }  i^{ X_{2}}_{\{1\}} ( (\mathfrak{A}_{i} \otimes \K)[1] ) \to 0
\end{align*}
is a semi-split extension of $C^{*}$-algebras over $X_{2}$ (see Definition~3.5 of \cite{rmrn:bootstrap}).  Set 
\begin{align*}
\mathfrak{I}_{i} =  i^{ X_{2} }_{\{2\}} ( ( \mathfrak{A}_{i} \otimes \K ) [ 2] ) \qquad \text{and} \qquad \mathfrak{B}_{i} = i^{ X_{2}}_{\{1\}} ( (\mathfrak{A}_{i} \otimes \K)[1] ).
\end{align*}   
By Theorem~3.6 of \cite{rmrn:bootstrap} (see also Korollar~3.4.6 \cite{bonkat})
\begin{align*}
\xymatrix{ \kk( X_{2} ; \mathfrak{B}_{1} , \mathfrak{A}_{2} \otimes \K ) \ar[r]^-{  ( \beta_{1} )^{*}} & \kk( X_{2}  ; \mathfrak{A}_{1} \otimes \K , \mathfrak{A}_{2} \otimes \K ) \ar[r]^-{ ( \lambda_{1} )^{*} } & \kk ( X_{2} ; \mathfrak{I}_{1} , \mathfrak{A}_{2} \otimes \K )
}
\end{align*}
is exact.  By Proposition~3.12 of  \cite{rmrn:bootstrap}, $\kk ( X_{2} ; \mathfrak{I}_{1} , \mathfrak{A}_{2} \otimes \K )$ and  $\kk (  \mathfrak{A}_{1} [2] \otimes \K , \mathfrak{A}_{2}[2] \otimes \K )$ are naturally isomorphic.   Hence, there exists $z \in \kk( X_{2} ; \mathfrak{B}_{1} , \mathfrak{A}_{2} \otimes \K )$ such that $y - \kk( X ; \psi ) = \kk ( X_{2} ; \beta_{1} ) \times z$ since $\kk( \psi_{ \{2\} } ) = r_{X_{2}}^{\{2\}} (y)$.  By Proposition~3.13 of \cite{rmrn:bootstrap},  $\kk( X_{2} ; \mathfrak{B}_{1} , \mathfrak{A}_{2} \otimes \K )$ and $\kk ( ( \mathfrak{A}_{1} \otimes \K )[1], \mathfrak{A}_{2} \otimes \K ) $ are isomorphic.  Therefore, by Theorem~8.3.3 of \cite{classentropy}, there exists a homomorphism $\ftn{ \eta }{ (\mathfrak{A}_{1}  \otimes \K)[1] }{ \mathfrak{A}_{2} \otimes \K }$ such that $\kk ( \eta ) = \overline{z}$, where $\overline{z}$ is the image of $z$ under the isomorphism $\kk( X_{2} ; \mathfrak{B}_{1} , \mathfrak{A}_{2} \otimes \K ) \cong \kk ( ( \mathfrak{A}_{1} \otimes \K )[1], \mathfrak{A}_{2} \otimes \K ) $ (the existence of the homomorphism uses the fact that $\mathfrak{A}_{2} \otimes \K$ is a properly infinite $C^{*}$-algebra which follows from Proposition~3.21 and Theorem~3.22 of \cite{err:fullext}).  Note that $\eta$ induces an $X_{2}$-equivariant homomorphism $\ftn{ \eta }{ \mathfrak{B}_{1} }{ \mathfrak{A}_{2} \otimes \K }$ such that $\kk ( X_{2} ; \eta ) = z$.  

Set $\phi = \psi + ( \eta \circ \beta_{1} )$, where the sum is the Cuntz sum in $\multialg{ \mathfrak{A}_{2} \otimes \K }$.  Then $\phi$ is an $X_{2}$-equivariant homomorphism such that $\kk ( X_{2} ; \phi ) = y$.  Since $\psi_{\{2\}}$ and $\psi_{\{1\}}$ are injective homomorphisms, $\phi_{\{2\}}$ and $\phi_{\{1\}}$ are injective homomorphisms. 
\end{proof}

\subsection{Strong classification of extensions of AF-algebras by purely infinite $C^{*}$-algebras}

\begin{defin}
Let $\mathfrak{A}$ and $\mathfrak{B}$ be separable $C^{*}$-algebras over $X$.  Two $X$-equivariant homomorphisms $\ftn{ \phi , \psi }{ \mathfrak{A} }{ \mathfrak{B} }$ are said to be \emph{approximately unitarily equivalent} if there exists a sequence of unitaries $\{ u_{n} \}_{ n = 1}^{ \infty }$ in $\multialg{ \mathfrak{B} }$ such that 
\begin{align*}
\lim_{ n \to \infty } \| u_{n} \phi ( a ) u_{n}^{*} - \psi ( a ) \| = 0 
\end{align*}
for all $a \in \mathfrak{A}$.
\end{defin}
 
We now recall the definition of $\kl ( \mathfrak{A} , \mathfrak{B} )$ from \cite{defKL}.  

\begin{defin}
Let $\mathfrak{A}$ be a separable, nuclear $C^{*}$-algebra in $\mathcal{N}$ and let $\mathfrak{B}$ be a $\sigma$-unital $C^{*}$-algebra.  Let 
\begin{align*}
\mathrm{Ext}_{\Z}^{1} ( K_{*} ( \mathfrak{A} ) , K_{*+1} ( \mathfrak{B} ) ) = \mathrm{Ext}_{\Z}^{1} ( K_{0} ( \mathfrak{A} ) , K_{1} ( \mathfrak{B} ) ) \oplus \mathrm{Ext}_{\Z}^{1}  ( K_{1} ( \mathfrak{A} ) , K_{0} ( \mathfrak{B} ) ).
\end{align*}
Since $\mathfrak{A}$ is in $\mathcal{N}$, by \cite{uct}, $\mathrm{Ext}_{\Z}^{1} ( K_{*} ( \mathfrak{A} ) , K_{*+1} ( \mathfrak{B} ) )$ can be identified as a sub-group of the group $\kk ( \mathfrak{A} , \mathfrak{B} )$.  

For abelian groups, $G$ and $H$, let $\mathrm{Pext}_{\Z}^{1} ( G , H )$ be the subgroup of $\mathrm{Ext}_{\Z}^{1} ( G , H )$ of all pure extensions of $G$ by $H$.  Set 
\begin{align*}
\mathrm{Pext}_{\Z}^{1} ( K_{*} ( \mathfrak{A} ) , K_{*+1} ( \mathfrak{B} ) ) = \mathrm{Pext}_{\Z}^{1} ( K_{0} ( \mathfrak{A} ) , K_{1} ( \mathfrak{B} ) ) \oplus \mathrm{Pext}_{\Z}^{1} ( K_{1} ( \mathfrak{A} ) , K_{0} ( \mathfrak{B} ) ).
\end{align*}
Define $\kl ( \mathfrak{A} , \mathfrak{B} )$ as the quotient
\begin{align*}
\kl ( \mathfrak{A} , \mathfrak{B} ) = \kk ( \mathfrak{A} , \mathfrak{B} ) / \mathrm{Pext}_{\Z}^{1} ( K_{*} ( \mathfrak{A} ) , K_{*+1} ( \mathfrak{B} ) ).
\end{align*}  
R{\o}rdam in \cite{defKL} proved that if $\ftn{ \phi , \psi }{ \mathfrak{A} }{ \mathfrak{B} }$ are approximately unitarily equivalent, then $\kl ( \phi ) = \kl ( \psi )$.
\end{defin}

\begin{notat}
Let $x \in \kk ( \mathfrak{A} , \mathfrak{B} )$.  Then the element $x + \mathrm{Pext}_{\Z}^{1} ( K_{*} ( \mathfrak{A} ) , K_{*+1} ( \mathfrak{B} ) )$ in $\kl ( \mathfrak{A} , \mathfrak{B} )$ will be denoted by $\kl ( x )$.  
\end{notat}

A nuclear, purely infinite, separable, simple $C^{*}$-algebra will be called a \emph{Kirchberg algebra}.

\begin{theor}(Uniqueness Theorem~1) \label{t:uniq1}
Let $\mathfrak{A}_{1}$ and $\mathfrak{A}_{2}$ be separable, nuclear, $C^{*}$-algebras over $X_{2}$ such that $\mathfrak{A}_{i}$ has real rank zero, $\mathfrak{A}_{i}$ is stable, $\mathfrak{A}_{i} [2]$ is a Kirchberg algebra in $\mathcal{N}$, $\mathfrak{A}_{i} [1]$ is an AF-algebra, and $\mathfrak{A}_{i}[2]$ is an essential ideal of $\mathfrak{A}_{i}$.  Suppose $\ftn{ \phi, \psi }{ \mathfrak{A}_{1} }{ \mathfrak{A}_{2} }$ be $X_{2}$-equivariant homomorphism such that $\kk ( X_{2} ; \phi ) = \kk ( X_{2} ; \psi )$, and $\phi_{\{2\}}$, $\phi_{\{1\}}$, $\psi_{\{2\}}$, and $\psi_{\{1\}}$ are injective homomorphisms.  Then $\phi$ and $\psi$ are approximately unitarily equivalent.
\end{theor}

\begin{proof}
Since  $\mathfrak{A}_{i} [ 1 ]$ is an AF algebra, every finitely generated subgroup of $K_{0} ( \mathfrak{A}_{i} [ 1] )$ is torsion free (hence free) and every finitely generated subgroup of $K_{1} ( \mathfrak{A}_{i} [ 1] )$ is zero.  Thus,  $\mathrm{Pext}_{\Z}^{1} ( K_{*} ( \mathfrak{A}_{i} [ 1] ) , K_{*+1} ( \corona{ \mathfrak{A}_{j} [ 2] } ) ) = \mathrm{Ext}_{\Z}^{1} ( K_{*} ( \mathfrak{A}_{i} [ 1] ) , K_{*+1} ( \corona{ \mathfrak{A}_{j} [ 2] } ) )$  which implies that $\kl ( \mathfrak{A}_{i} [ 1 ] , \corona{ \mathfrak{A}_{j} [ 2] } ) \cong \Hom ( K_{*} ( \mathfrak{A}_{i} [ 1] ) , K_{*} ( \corona{ \mathfrak{A}_{j} [ 2] } ) )$.

Let $\mathfrak{e}_{i}$ denote the extension $0 \to \mathfrak{A}_{i} [ 2 ]  \to \mathfrak{A}_{i}  \to  \mathfrak{A}_{i} [ 1] \to 0$.  Since $\mathfrak{A}_{i}$ has real rank zero and $K_{1} ( \mathfrak{A}_{i }[1] ) = 0$, we have that $K_{j} ( \tau_{ \mathfrak{e}_{i} } ) = 0$, where $\tau_{ \mathfrak{e}_{i}}$ is the Busby invariant of $\mathfrak{e}_{i}$.  Hence, $[ \tau_{ \mathfrak{e}_{i} } ] = 0$ in $\kl ( \mathfrak{A}_{i} [ 1 ] , \corona{ \mathfrak{A}_{i} [ 2 ] } )$.  By Corollary~6.7 of \cite{quasidiagext}, $\mathfrak{e}_{i}$ is quasi-diagonal.  Thus, there exists an approximate identity of $\mathfrak{A}_{1} [ 2 ]$ consisting of projections $\{ e_{k} \}_{k \in \N }$ such that 
\begin{align*}
\lim_{ n \to \infty } \| e_{k} x - x e_{k} \| = 0
\end{align*}
for all $x \in  \mathfrak{A}_{i}$.

Since $\mathfrak{A}_{1}[1]$ is an AF-algebra and $\mathfrak{A}_{1}$ has real rank zero, as in the proof of Lemma~9.8 of \cite{ee:dim}, there exists a sequence of finite dimensional sub-$C^{*}$-algebras $\{ \mathfrak{B}_{k} \}_{ k = 1}^{ \infty }$ of $\mathfrak{A}_{1}$ such that $\mathfrak{B}_{k} \cap \mathfrak{A}_{1} [ 2 ] = \{ 0 \}$ and for each $x \in \mathfrak{A}_{1}$, there exist $y_{1} \in \overline{ \bigcup_{ k = 1}^{ \infty } \mathfrak{B}_{k} }$ and $y_{2} \in \mathfrak{A}_{1} [ 2 ]$ such that  $x = y_{1} + y_{2}$.

Let $\epsilon > 0$ and $\mathcal{F}$ be a finite subset of $\mathfrak{A}_{1}$.  Note that we may assume $\mathcal{F}$ is the union of the generators of $\mathfrak{B}_{m}$, for some $m \in \N$ and $\mathcal{G}$, for some finite subset $\mathcal{G}$ of $\mathfrak{A}_{1} [ 2 ]$ .  Since $\mathfrak{B}_{m}$ is a finite dimensional $C^{*}$-algebra,
\begin{align*}
\lim_{ k \to \infty } \| e_{k} x - x e_{k} \| = 0
\end{align*}
for all $x \in \mathfrak{A}_{1}$, and $\{ e_{k} \}_{k \in \N }$ is an approximate identity for $\mathfrak{A}_{1} [ 2 ]$ consisting of projections, there exist $k \in \N$, a finite dimensional sub-$C^{*}$-algebra $\mathfrak{D}$ of $\mathfrak{A}_{1}$ with $\mathfrak{D} \subseteq ( 1_{ \multialg{ \mathfrak{A}_{1} } } - e_{k} ) \mathfrak{A}_{1} ( 1_{ \multialg{ \mathfrak{A}_{1} } } - e_{k} )$ and $\mathfrak{D} \cap \mathcal{A}_{1} [ 2 ] = \{ 0 \}$, and there exists a finite subset $\mathcal{H}$ of $e_{k} \mathfrak{A}_{1} [ 2 ] e_{k}$ such that for all $x \in \mathcal{F}$, there exist $y_{1} \in \mathfrak{D}$ and $y_{2} \in \mathcal{H}$
\begin{align*}
\| x - ( y_{1} + y_{2} ) \| < \frac{ \epsilon }{ 3 }.
\end{align*} 

Set $\mathfrak{D} = \bigoplus_{ \ell = 1}^{s } \mathsf{M}_{n_{\ell}}$ and let $\{ f_{ij}^{\ell} \}_{ i , j = 1}^{ n_{ \ell }}$ be a system of matrix units for $\mathsf{M}_{n_{ \ell } }$.  Let $\mathfrak{I}_{\ell}$ be the ideal in $\mathfrak{A}_{1}$ generated by $f_{11}^{\ell}$.  Since $\mathfrak{A}_{i}[2]$ is simple and $\mathfrak{A}_{i}[2]$ is an essential ideal of $\mathfrak{A}_{i}$, we have that $\mathfrak{A}_{i} [ 2 ] \subseteq \mathfrak{I}$ for all nonzero ideal $\mathfrak{I}$ of $\mathfrak{A}_{i}$.  Thus, $\mathfrak{A}_{1} [2] \subseteq \mathfrak{I}_{\ell}$ since $\mathfrak{D} \cap \mathfrak{A}_{1} [ 2] = 0$.

Let $\mathfrak{I}_{\ell}^{\phi}$ be the ideal in $\mathfrak{A}_{2}$ generated by $\phi( f_{11}^{\ell} )$ and let $\mathfrak{I}_{\ell}^{\psi}$ be the ideal in $\mathfrak{A}_{2}$ generated by $\psi( f_{11}^{\ell} )$.  Since $\phi$ and $\psi$ are $X_{2}$-equivariant homomorphisms and since $\phi_{\{1\}}$ and $\psi_{\{1\}}$ are injective homomorphisms, we have that $\phi( f_{11}^{\ell}) \notin \mathfrak{A}_{2}[2]$ and $\psi( f_{11}^{\ell} ) \notin \mathfrak{A}_{2} [ 2 ]$.  Therefore, $\mathfrak{A}_{2} [ 2 ] \subseteq \mathfrak{I}_{\ell}^{\phi}$ and $\mathfrak{A}_{2} [ 2 ] \subseteq \mathfrak{I}_{\ell}^{\psi}$.  Since $K_{0} ( \phi_{\{1\}} ) = K_{0} ( \psi_{\{1\}} )$ and since $\mathfrak{A}_{2} [ 1 ]$ is an AF-algebra, we have that $\phi_{\{1\}} ( \overline{f}_{11}^{\ell} )$ is Murray-von Neumann equivalent to $\psi_{\{1\}} ( \overline{f}_{11}^{\ell} )$, where $\overline{f}_{11}^{\ell}$ is the image of $f_{11}^{\ell}$ in $\mathfrak{A}_{1} [ 1 ]$.  Thus, they generate the same ideal in $\mathfrak{A}_{2}[1]$.  Since $\mathfrak{A}_{2} [ 2 ] \subseteq \mathfrak{I}_{\ell}^{\phi}$ and $\mathfrak{A}_{2} [ 2 ] \subseteq \mathfrak{I}_{\ell}^{\psi}$ and since $\psi_{\{1\}} ( \overline{f}_{11}^{\ell} )$ and $\phi_{ \{1\}} ( \overline{f}_{11}^{\ell})$ generate the same ideal in $\mathfrak{A}_{2}[1]$, we have that $\mathfrak{I} = \mathfrak{I}_{\ell}^{\phi} = \mathfrak{I}_{\ell}^{\psi}$.

Note that the following diagram
\begin{align*}
\xymatrix{
0 \ar[r] & K_{0} ( \mathfrak{A}_{2}[2] ) \ar[r] \ar@{=}[d] & K_{0} ( \mathfrak{I} ) \ar[r] \ar[d]^{K_{0} (\iota) } & K_{0} ( \mathfrak{I} / \mathfrak{A}_{2}[2] ) \ar[d]^{K_{0} (\overline{\iota})} \\
0 \ar[r] & K_{0} ( \mathfrak{A}_{2}[2] ) \ar[r] & K_{0} ( \mathfrak{A}_{2} ) \ar[r] & K_{0} ( \mathfrak{A}_{2}[1]  )
}
\end{align*}
is commutative, the rows are exact, and $\iota$ and $\overline{\iota}$ are the canonical embeddings.  Since $\mathfrak{A}_{2}[1]$ is an AF-algebra, $K_{0} ( \overline{\iota} )$ is injective.  A diagram chase shows that $K_{0} ( \iota )$ is injective.  Since $\kk ( X_{2} ; \phi ) = \kk ( X_{2} ; \psi )$, we have that $[ \phi( f_{11}^{\ell} ) ] = [ \psi ( f_{11}^{ \ell} ) ]$ in $K_{0} ( \mathfrak{A}_{2} )$.  Since $\phi( f_{11}^{\ell} )$ and $\psi( f_{11}^{\ell} )$ are elements of $\mathfrak{I}$ and $K_{0} ( \iota )$ is injective, we have that $[ \phi( f_{11}^{\ell} ) ] = [ \psi ( f_{11}^{ \ell} ) ]$ in $K_{0} ( \mathfrak{I})$.  Since $\mathfrak{A}_{i} [ 1]$ is an AF-algebra and $\mathfrak{A}_{i}[2]$ is a Kirchberg algebra, they both have stable weak cancellation.  By Lemma~3.15 of \cite{err:fullext}, $\mathfrak{A}_{i}$ has stable weak cancellation.  Thus, $\phi ( f_{11}^{\ell} )$ is Murray-von Neumann equivalent to $\psi ( f_{11}^{ \ell } )$.  Hence, there exists $v_{ \ell } \in \mathfrak{A}_{2}$ such that $v_{\ell}^{*} v_{\ell} = \phi ( f_{11}^{\ell} )$ and $v_{\ell} v_{\ell}^{*} = \psi ( f_{11}^{ \ell } )$.  

Set 
\begin{align*}
u_{1} = \sum_{ \ell = 1}^{s} \sum_{ i =1}^{n_{ \ell } } \psi ( f_{i1}^{\ell} ) v_{\ell} \phi ( f_{1i}^{\ell} )
\end{align*}  
Then, $u_{1}$ is a partial isometry in $\mathfrak{A}_{1}$ such that $u_{1}^{*} u_{1} = \phi ( 1_{ \mathfrak{D} } )$, $u_{1} u_{1}^{*} = \psi ( 1_{ \mathfrak{D} } )$, and $u_{1} \phi ( x ) u_{1}^{*} = \psi (x)$ for all $x \in \mathfrak{D}$.

Let $\ftn{ \beta }{ e_{k} \mathfrak{A}_{1} [ 2 ] e_{k} }{ \mathfrak{A}_{1} [ 2 ] }$ be the usual embedding.  Note that $\kk ( \phi_{ \{ 2 \}  } \circ \beta ) = \kk ( \psi_{ \{ 2 \} } \circ \beta )$ and $\phi_{ \{ 2 \} } \circ \beta $, $\psi_{ \{ 2 \} } \circ \beta$ are monomorphisms.  Therefore, by Theorem~6.7 of \cite{sepBDF}, there exists a partial isometry $u_{2} \in \mathfrak{A}_{2} [ 2 ]$ such that $u_{2}^{*} u_{2} = \phi ( e_{k} )$, $u_{2} u_{2}^{*} = \psi ( e_{k} )$, and
\begin{align*}
\| u_{2} \phi(x) u_{2}^{*} - \psi ( x ) \| < \frac{ \epsilon }{ 3 }
\end{align*} 
for all $x \in \mathcal{H}$.  

Since $\mathfrak{A}_{2}$ is stable, there exists $u_{3} \in \multialg{ \mathfrak{A}_{2} }$ such that $u_{3}^{*} u_{3} = 1_{ \multialg{ \mathfrak{A}_{2} } } - ( u_{1} + u_{2} )^{*} ( u_{1} + u_{2} )$ and $u_{3} u_{3}^{*} = 1_{ \multialg{ \mathfrak{A}_{2} } } - ( u_{1} + u_{2} )( u_{1} + u_{2} )^{*}$.  Set $u = u_{1} + u_{2} + u_{3} \in \multialg{ \mathfrak{A}_{2} }$.  Then $u$ is a unitary in $\multialg{ \mathfrak{A}_{2} }$.  

Let $x \in \mathcal{F}$.  Choose $y_{1} \in \mathfrak{D}$ and $y_{2} \in \mathcal{H} \subseteq e_{k} \mathfrak{A}_{1} [ n] e_{k}$ such that $\| x - ( y_{ 1 } + y_{2} ) \| < \frac{ \epsilon }{3} $.  Then
\begin{align*}
\| u \phi ( x ) u^{*} - \psi ( x ) \| &\leq \| u \phi (x) u^{*} - u \phi( y_{1} + y_{2} ) u^{*} \|  \\
					&\qquad   + \| u_{1} \phi ( y_{1} ) u_{1}^{*} + u_{2} \phi ( y_{2} ) u_{2}^{*} - \psi ( y_{1} ) - \psi ( y_{2} ) \| \\
					&\qquad \qquad  + \| \psi ( y_{1} + y_{2} ) - \psi ( x ) \|   \\
						&< \epsilon.
\end{align*}

We have just shown that for each $\epsilon > 0$ and for each finite subset $\mathcal{F}$ of $\mathfrak{A}_{1}$, there exists a unitary $u \in \multialg{ \mathfrak{A}_{2} }$ such that $\| u \phi ( x ) u^{*} - \psi ( x ) \| < \epsilon$ for all $x \in \mathcal{F}$.  Since $\mathfrak{A}_{1}$ is a separable $C^{*}$-algebra, we have that $\phi$ is approximately unitarily equivalent to $\psi$.
\end{proof}

\begin{lemma}\label{l:unitaryidealkthy}
Let $\mathfrak{A}$ be a separable $C^{*}$-algebra over a finite topological space $X$.  Let $u$ be unitary in $\multialg{ \mathfrak{A} \otimes \K }$.  Then $K_{X}\left( \mathrm{Ad} ( u ) \vert_{ \mathfrak{A} \otimes \K } \right) = \id_{ K_{X} ( \mathfrak{A} )} $.
\end{lemma}

\begin{proof}
Since $\mathfrak{A} \otimes \K$ is stable, we have that there exists a norm continuous path of unitaries $\{ u_{t} \}$ in $\multialg{ \mathfrak{A} \otimes \K }$ such that $u_{0} = u$ and $u_{1} = 1_{ \multialg{ \mathfrak{A} \otimes \K } }$.  It follows that $K_{X}\left( \mathrm{Ad} ( u ) \vert_{ \mathfrak{A} \otimes \K } \right) = \id_{ K_{X} ( \mathfrak{A} )}$. 
\end{proof}

\begin{theor}\label{t:classmixed1}
Let $\mathfrak{A}_{1}$ and $\mathfrak{A}_{2}$ be in $\mathcal{B}( X_{2})$ and let $x \in \kk( X_{2}; \mathfrak{A}_{1} , \mathfrak{A}_{2} )$ be an invertible element such that $\Gamma(x)_{Y}$ is an order isomorphism for all $Y \in \mathbb{LC}(X_{2})$.  Suppose $\mathfrak{A}_{i} [ 2 ]$ is a Kirchberg algebra, $\mathfrak{A}_{i} [ 1]$ is an AF-algebra, $\mathfrak{A}_{i}$ has real rank zero, and $\mathfrak{A}_{i}[2]$ is an essential ideal of $\mathfrak{A}_{i}$.  Then there exists an $X_{2}$-equivariant isomorphism $\ftn{ \phi }{ \mathfrak{A}_1 \otimes \K }{ \mathfrak{A}_2 \otimes \K }$ such that $\kl ( \phi ) = \kl ( g_{X_{2}}^{1}(y) )$ and $K_{X_{2}} ( \phi ) = K_{X_{2}} ( y )$, where $y = \kk( X_{2} ; \id_{ \mathfrak{A}_{1} } \otimes e_{11} )^{-1} \times x \times \kk( X_{n} ; \id_{ \mathfrak{A}_{2} } \otimes e_{11} )$
\end{theor}

\begin{proof}
Since $\mathfrak{A}_{i}[2]$ is a purely infinite simple $C^{*}$-algebra, $\mathfrak{A}_{i} [ 2 ]$ is either unital or stable.  Since $\mathfrak{A}_{i} [ 2 ]$ is an essential ideal of $\mathfrak{A}_{i}$, $\mathfrak{A}_{i}[2]$ is non-unital else $\mathfrak{A}_{i} [ 2]$ is isomorphic to a direct summand of $\mathfrak{A}_{i}$ which would contradict the essential assumption.  Therefore, $\mathfrak{A}_{i}[2]$ is stable.  Moreover, $\corona{ \mathfrak{A}_{i} [ 2 ] }$ is simple which implies that $0 \to \mathfrak{A}_{i} [ 2 ] \to \mathfrak{A}_{i} \to \mathfrak{A}_{i}[1] \to 0$ is a full extension.  Since $\mathfrak{A}_{i}[2]$ and $\mathfrak{A}_{i}[1]$ are nuclear $C^{*}$-algebras, $\mathfrak{A}_{i}$ is a nuclear $C^{*}$-algebra.

Let $z \in \kk( X_{2} ; \mathfrak{A}_2 \otimes \K , \mathfrak{A}_1 \otimes \K )$ such that $y \times z = [ \id_{ \mathfrak{A}_{1} \otimes \K } ]$ and $y \times z = [ \id_{ \mathfrak{A}_2 \otimes \K } ]$.  By Theorem~\ref{t:exist}, there exists an $X_{2}$-equivariant homomorphism $\ftn{ \psi_{1} }{ \mathfrak{A}_1 \otimes \K }{ \mathfrak{A}_2 \otimes \K }$ such that $\kk ( X_{2} ; \psi_{1} ) = x$, and $( \psi_{1} )_{ \{ 2 \} }$ and $( \psi_{1} )_{ \{ 1 \} }$ are injective homomorphisms.  By Theorem~\ref{t:exist}, there exists an $X_{2}$-equivariant homomorphism $\ftn{ \psi_{2} }{ \mathfrak{A}_2 \otimes \K }{ \mathfrak{A}_1 \otimes \K }$ such that $\kk ( X_{2} ; \psi_{2} ) = y$, and $( \psi_{2} )_{\{2\}}$ and $( \psi_{2} )_{\{1\}}$ are injective homomorphisms.  Using Theorem~\ref{t:uniq1} and a typical approximate intertwining argument, there exists an isomorphism $\ftn{ \phi }{ \mathfrak{A}_1 \otimes \K }{ \mathfrak{A}_2 \otimes \K }$ such that $\phi$ and $\psi_{1}$ are approximately unitarily equivalent.  

Let $\ftn { \pi_{2} }{ \mathfrak{A}_{2} }{ \mathfrak{A}_{2} [ 1] }$ be the canonical quotient map.  Then $\pi_{2} \circ \phi \vert_{ \mathfrak{A}_{1} [ 2 ] }$ is either zero or injective since $\mathfrak{A}_{1} [ 2 ]$ is simple.  Since $\mathfrak{A}_{1} [ 2 ]$ is purely infinite and $\mathfrak{A}_{2} [1]$ is an AF-algebra, we must have that $\pi_{2} \circ \phi \vert_{ \mathfrak{A}_{2}[2] } = 0$.  Thus, $\phi$ is an $X_{2}$-equivariant homomorphism.  Similarly, $\phi^{-1}$ is an $X_{2}$-equivariant homomorphism.  Hence, $\phi$ is an $X_{2}$-equivariant isomorphism.  By construction, $\kl ( \phi ) = \kl ( \psi_{1} ) =  \kl ( g_{X_{2}}^{1}(y) )$.  By Lemma~\ref{l:unitaryidealkthy}, $K_{X_{2}} ( \phi ) = K_{X_{n}} (x)$.
\end{proof}

\begin{corol}\label{c:classmixed1}
Let $\mathfrak{A}_{1}$ and $\mathfrak{A}_{2}$ be in $\mathcal{B}( X_{2})$ and let $x \in \kk( X_{2}; \mathfrak{A}_{1} , \mathfrak{A}_{2} )$ be an invertible element such that $\Gamma(x)_{Y}$ is an order isomorphism for all $Y \in \mathbb{LC}(X_{2})$.  Suppose $\mathfrak{A}_{i} [ 2 ]$ is a Kirchberg algebra, $\mathfrak{A}_{i} [ 1]$ is an AF-algebra, $\mathfrak{A}_{i}$ has real rank zero, $\mathfrak{A}_{i}[2]$ is an essential ideal of $\mathfrak{A}_{i}$, and $K_{i} ( \mathfrak{A} [ Y ] )$ and $K_{i} ( \mathfrak{B}[ Y ] )$ are finitely generated for all $Y \in \mathbb{LC} ( X_{2} )$.  Then there exists an $X_{2}$-equivariant isomorphism $\ftn{ \phi }{ \mathfrak{A}_1 \otimes \K }{ \mathfrak{A}_2 \otimes \K }$ such that $\kk ( \phi ) = \kk ( g_{X_{2}}^{1}(y) )$ and $K_{X_{2}} ( \phi ) = K_{X_{2}} ( y )$, where $y = \kk( X_{2} ; \id_{ \mathfrak{A}_{1} } \otimes e_{11} )^{-1} \times x \times \kk( X_{n} ; \id_{ \mathfrak{A}_{2} } \otimes e_{11} )$
\end{corol}

\begin{proof}
This follows from Theorem~\ref{t:classmixed1} and the fact that if $G$ is finitely generated, then $\mathrm{Pext}_{\Z}^{1} ( G , H ) = 0$.
\end{proof}

\subsection{Strong classification of extensions of purely infinite by $\K$}

We recall the following from \cite[p. 341]{wa:extensions}.  Let $\ftn{ \psi }{ \mathfrak{A} }{ B( \mathcal{H} ) }$ be a representation of $\mathfrak{A}$.  Let $\mathcal{H}_{e}$ denote the subspace of $\mathcal{H}$ spanned by the ranges of all compact operators in $\psi(\mathfrak{A})$.  Since $\psi( \mathfrak{A}) \cap \K$ is an ideal of $\psi( \mathfrak{A} )$, we have that $\mathcal{H}_{e}$ reduces $\pi( \mathfrak{A} )$, and so the decomposition $\mathcal{H} = \mathcal{H}_{e} \oplus \mathcal{H}_{e}^{\perp}$ induces a decomposition of $\psi$ into sub-representations $\psi = \psi_{e} \oplus \psi'$.  The summand $\psi_{e}$, considered as a representation of $\mathfrak{A}$ on $\mathcal{H}_{e}$, will be called the \emph{essential part of $\psi$} and $\mathcal{H}_{e}$ is called the \emph{essential subspace} for $\psi$.

Let $\mathfrak{B}$ be a tight $C^{*}$-algebra over $X_{2}$.  Consider the essential extension
\begin{align*}
\mathfrak{e}_{ \mathfrak{B} }: 0 \to \mathfrak{B} [2] \to \mathfrak{B} \to \mathfrak{B} [1] \to 0.
\end{align*}
If $\ftn{ \tau_{\mathfrak{e}_{ \mathfrak{B} }} }{ \mathfrak{B}[1] }{ \corona{ \mathfrak{B}[2] } }$ is the Busby invariant of $\mathfrak{e}$, then there exists an injective homomorphism $\ftn{ \sigma_{\mathfrak{e}_{ \mathfrak{B} }} }{ \mathfrak{B} }{ \multialg{ \mathfrak{B}[2] } }$ such that the diagram
\begin{align*}
\xymatrix{
0 \ar[r] & \mathfrak{B}[2] \ar[r] \ar@{=}[d] & \mathfrak{B} \ar[r]^{ \pi_{ \mathfrak{B} } } \ar[d]^{ \sigma_{ \mathfrak{e}_{ \mathfrak{B} } } } & \mathfrak{B}[1] \ar[r] \ar[d]^{ \tau_{\mathfrak{e}_{ \mathfrak{B} } } } & 0 \\
0 \ar[r] & \mathfrak{B}[2] \ar[r] & \multialg{ \mathfrak{B}[2] } \ar[r]_{\overline{\pi}_{ \mathfrak{B} }  } & \corona{ \mathfrak{B}[2] } \ar[r] & 0
}
\end{align*}
If $\mathfrak{B}[2] \cong \K$, let $\ftn{ \eta_{\mathfrak{B}} }{ \multialg{ \mathfrak{B} [ 2 ]  } }{ B( \ell^{2} ) }$ be the isomorphism extending the isomorphism $\mathfrak{B}[2] \cong \K$ and let $\ftn{ \overline{\eta}_{\mathfrak{B}} }{ \corona{ \mathfrak{B} [ 2 ]  } }{ B( \ell^{2} ) /\K}$ be the induced isomorphism.

\begin{lemma}\label{l:rep}
Let $\mathfrak{A}$ and $\mathfrak{B}$ be separable, tight $C^{*}$-algebras over $X_{2}$ such that $\mathfrak{A}[2] \cong \mathfrak{B}[2] \cong \K$.  Let $\ftn{ \psi_{1}, \psi_{2} }{ \mathfrak{A} }{ \mathfrak{B} }$ be two, full $X_{2}$-equivariant homomorphisms such that $K_{0}( ( \psi_{1})_{\{2\} } ) = K_{0} ( ( \psi_{2} )_{ \{ 2 \} } )$ and $\eta_{ \mathfrak{B} } \circ \sigma_{ \mathfrak{e}_{ \mathfrak{B} } } \circ \psi_{i}$ is a non-degenerate representation of $\mathfrak{A}$.  Then there exists a sequence of unitaries $\{ U_{n} \}_{n = 1}^{ \infty }$ in $\multialg{ \mathfrak{B} [2] }$ such that 
\begin{align*}
U_{n} ( \sigma_{\mathfrak{e}_{\mathfrak{B} } } \circ \psi_{1} )( a ) U_{n}^{*} - ( \sigma_{\mathfrak{e}_{\mathfrak{B} } } \circ \psi_{2} ) ( a ) \in \mathfrak{B}[2]
\end{align*}
for all $a \in \mathfrak{A}$ and for all $n \in \N$, and 
\begin{align*}
\lim_{ n \to \infty } \| U_{n} (\sigma_{\mathfrak{e}_{\mathfrak{B} } } \circ \psi_{1} )( a ) U_{n}^{*} - (\sigma_{\mathfrak{e}_{\mathfrak{B} } } \circ \psi_{2} )( a ) \| = 0
\end{align*}
for all $a \in \mathfrak{A}$.
\end{lemma}

\begin{proof}
We argue as in the proof of Lemma~2.8 of \cite{LinK1zero}.  Set $\sigma_{i} = \eta_{\mathfrak{B} } \circ \sigma_{ \mathfrak{e}_{ \mathfrak{B} } } \circ \psi_{i}$.  By assumption, $\ftn{ \sigma_{i} }{ \mathfrak{A} }{ B( \ell^{2} ) }$ is a non-degenerated representation of $\mathfrak{A}$.  We claim that there exists a sequence of unitaries $\{ V_{n} \}_{ n = 1}^{ \infty }$ in $B( \ell^{2} )$ such that $V_{n} \sigma_{1} ( a ) V_{n}^{*} - \sigma_{2} ( a ) \in \K$ for all $n \in \N$ and 
\begin{align*}
\lim_{ n \to \infty } \| V_{n} \sigma_{1} ( a ) V_{n}^{*} - \sigma_{2} ( a ) \| = 0
\end{align*}
for all $a \in \mathfrak{A}$.  This will be a consequence of Theorem~5(iii) of \cite{wa:extensions}.

Let $\ftn{ \rho }{ \mathfrak{A} }{ B( \ell^{2} ) }$ be the unique irreducible faithful representation defined by the isomorphism $\mathfrak{A}[2] \cong \K$.  Since $\psi_{i}, \sigma_{\mathfrak{e}_{ \mathfrak{B} }}, \eta_{ \mathfrak{B} }$ are injective homomorphisms, $\sigma_{i}$ is injective.  Therefore, $\ker ( \sigma_{1} ) = \ker ( \sigma_{2} ) = \{ 0 \}$.  Let $\ftn{ \pi }{ B( \ell^{2} ) }{ B( \ell^{2} ) / \K }$ be the natural projection.  Note that 
\begin{align*}
 \pi \circ \sigma_{i} = \pi \circ \eta_{\mathfrak{B} } \circ \sigma_{ \mathfrak{e}_{ \mathfrak{B} } } \circ \psi_{i}  = \overline{\eta}_{\mathfrak{B} } \circ \overline{\pi}_{\mathfrak{B} } \circ \sigma_{\mathfrak{B} } \circ \psi_{i} = \overline{\eta}_{\mathfrak{B} } \circ \tau_{ \mathfrak{e}_{\mathfrak{B} } } \circ \pi_{\mathfrak{B} } \circ \psi_{i} = \overline{\eta}_{\mathfrak{B} } \circ \tau_{ \mathfrak{e}_{\mathfrak{B} } } \circ ( \psi_{i} )_{\{1\}} \circ \pi_{\mathfrak{A} } .   
\end{align*}  
It now follows that $\ker ( \pi \circ \sigma_{1} ) = \ker( \pi \circ \sigma_{2} ) = \mathfrak{A}[2]$ since $\overline{\eta}_{\mathfrak{B}}$, $\tau_{ \mathfrak{e}_{\mathfrak{B}}}$, and $( \psi_{i} )_{\{1\}}$ are injective homomorphisms.  

Let $H_{1}$ be the essential subspace of $\sigma_{1}$.  Since $\sigma_{1} ( \mathfrak{A}[2] ) \subseteq \K$ and for each $x \notin \mathfrak{A}[2]$, we have that $\sigma_{1} ( x ) \notin \K$, we have that $H_{1} = \overline{ \sigma_{1} ( \mathfrak{A}[2] ) \ell^{2} }$.  Similarly, we have that $H_{2} = \overline{ \sigma_{2} ( \mathfrak{A}[2] ) \ell^{2} }$, where $H_{2}$ is the essential subspace of $\sigma_{2}$.  Let $e$ be a minimal projection of $\mathfrak{A}[2] \cong \K$.  Suppose $\sigma_{1} ( e )$ has rank $k$.  Standard representation theory now implies that $\sigma_{1} (- ) \vert_{H_{1}}$ is unitarily equivalent to the direct sum of $k$ copies of $\rho$.  Since $K_{0}( ( \psi_{1})_{\{2\} } ) = K_{0} ( ( \psi_{2} )_{ \{ 2 \} } )$, we have that $\sigma_{1} ( e )$ is Murray-von Neumann equivalent to $\sigma_{2} ( e )$.  Hence, $\sigma_{2} ( e )$ has rank $k$.  Standard representation theory now implies that $\sigma_{2} (- ) \vert_{H_{2}}$ is unitarily equivalent to the direct sum of $k$ copies of $\rho$.  

The above paragraph imply that $\sigma_{2} (-) \vert_{H_{2}}$ and $\sigma_{1} (-) \vert_{H_{1}}$ are unitarily equivalent.  Since $\ker( \sigma_{1} )  = \ker( \sigma_{2} )$ and $\ker( \pi \circ \sigma_{1} ) = \ker( \pi \circ \sigma_{2} )$ by Theorem~5(iii) of \cite{wa:extensions}, there exists a sequence of unitaries $\{ V_{n} \}_{ n = 1}^{ \infty }$ in $B( \ell^{2} )$ such that $V_{n} \sigma_{1} ( a ) V_{n}^{*} - \sigma_{2} ( a ) \in \K$ for all $n \in \N$ and for all $a \in \mathfrak{A}$,  and 
\begin{align*}
\lim_{ n \to \infty } \| V_{n} \sigma_{1} ( a ) V_{n}^{*} - \sigma_{2} ( a ) \| = 0
\end{align*}
for all $a \in \mathfrak{A}$.  

Set $U_{n} = \eta_{\mathfrak{B}}^{-1} ( V_{n} )$.  Then $\{ U_{n} \}_{n = 1}^{ \infty }$ is a sequence of unitaries in $\multialg{ \mathfrak{B} [2] }$ such that $U_{n} ( \sigma_{ \mathfrak{e}_{ \mathfrak{B} } } \circ \psi_{1} )( a ) U_{n}^{*} - ( \sigma_{\mathfrak{e}_{\mathfrak{B} } }  \circ \psi_{2} )( a ) \in \mathfrak{B}[2]$ for all $n \in \N$ and for all $a \in \mathfrak{A}$,  and 
\begin{align*}
\lim_{ n \to \infty } \| U_{n} ( \sigma_{\mathfrak{e}_{\mathfrak{B} } } \circ \psi_{1} ) ( a ) U_{n}^{*} - ( \sigma_{\mathfrak{e}_{\mathfrak{B} } } \circ \psi_{2} ) ( a ) \| = 0
\end{align*}
for all $a \in \mathfrak{A}$.      
\end{proof}

\begin{defin}
A $C^{*}$-algebra $\mathfrak{A}$ is called \emph{weakly semiprojective} if we can always solving the $*$-homomorphism lifting problem
\begin{align*}
\xymatrix{
		& \prod_{n = N}^{\infty } \mathfrak{B}_{n} \ar[d]^{\rho_{N}} &  & ( b_{N} , b_{N+1} , \dots ) \ar@{|->}[d]   \\
\mathfrak{A} \ar[r]_-{ \phi } \ar@{-->}[ru]^-{\widetilde{\phi}} & \prod_{ n = 1}^{\infty } \mathfrak{B}_{n} / \bigoplus_{ n = 1}^{ \infty } \mathfrak{B}_{n}  & & [ (0, \dots,0, b_{N} , b_{N+1} , \dots ) ]
}
\end{align*}
and $\mathfrak{A}$ is called \emph{semiprojective} if we can always solve the lifting problem
\begin{align*}
\xymatrix{
		& \mathfrak{B} / \mathfrak{I}_{N} \ar[d]^{\rho_{N}}   \\
\mathfrak{A} \ar[r]_-{ \phi } \ar@{-->}[ru]^-{\widetilde{\phi}} & \mathfrak{B} / \overline{ \bigcup_{ n = 1}^{ \infty } \mathfrak{I}_{n} } & & ( \mathfrak{I}_{1} \unlhd \mathfrak{I}_{2} \unlhd \cdots \unlhd \mathfrak{B}) 
}
\end{align*}
\end{defin}

\begin{lemma}\label{l:rep2}
Let $\mathfrak{A}_{0}$ be  a unital, separable, nuclear, tight $C^{*}$-algebra over $X_{2}$ such that $\mathfrak{A}_{0}[2] \cong \K$ and $\mathfrak{A}_{0}$ has the stable weak cancellation property.  Set $\mathfrak{A} = \mathfrak{A}_{0} \otimes \K$.  Suppose $\ftn{ \beta }{ \mathfrak{A} }{ \mathfrak{A} }$ is a full $X_{2}$-equivariant homomorphism such that $K_{X_{2}} ( \beta ) = K_{X_{2}} ( \id_{ \mathfrak{A} } )$ and $\beta_{ \{ 1 \} } = \id_{ \mathfrak{A}[1] }$.  Then there exists a sequence of contractive, completely positive, linear maps $\{ \ftn{ \alpha_{n} }{ \mathfrak{A}  }{ \mathfrak{A}} \}_{ n = 1}^{\infty}$ such that 
\begin{itemize}
\item[(1)] $\alpha_{n} \vert_{ e_{n} \mathfrak{A} e_{n} }$ is a homomorphism for all $n \in \N$ and 

\item[(2)] for all $a \in \mathfrak{A}$,
\begin{align*}
\lim_{ n \to \infty } \| \alpha_{n} \circ \beta (a)  - a \| = 0
\end{align*}
\end{itemize}
where $e_{n} = \sum_{ k = 1}^{n} 1_{ \mathfrak{A}_{0}} \otimes e_{kk}$ and $\{ e_{ij} \}_{i,j}$ is a system of matrix units for $\K$.
 If, in addition, $\mathfrak{A}$ is assumed to be weakly semiprojective, then $\alpha_{n}$ can be chosen to be a homomorphism for all $n \in \N$.
\end{lemma}

\begin{proof}
Since $\beta$ is a full $X_{2}$-equivariant homomorphism and the ideal in $\mathfrak{A}$ generated by $e_{n}$ is $\mathfrak{A}$, we have that the ideal in $\mathfrak{A}$ generated by $\beta( e_{n} )$ is $\mathfrak{A}$.  Since $K_{X_{2}} ( \beta ) = K_{X_{2}} ( \id_{ \mathfrak{A} } )$, we have that $[ \beta( e_{n} ) ] = [ e_{n} ]$ in $K_{0} ( \mathfrak{A} )$.  It now follows that $\beta ( e_{n} )$ and $e_{n}$ are Murray-von Neumann equivalent since $\mathfrak{A}_{0}$ has the stable weak cancellation property.  Since $\mathfrak{A}$ is stable, there exists a unitary $v_{n}$ in the unitization of $\mathfrak{A}$ such that $v_{n} \beta( e_{n} ) v_{n}^{*} = e_{n}$.  

Fix $n \in \N$.  Let $\mathfrak{e}_{ n }$ be the extension $0 \to e_{n} \mathfrak{A}[2] e_{n} \to e_{n} \mathfrak{A}e_{n} \to \overline{ e }_{n} \mathfrak{A}[1] \overline{e}_{n} \to 0$.  By Lemma~1.5 of \cite{ERRshift}, $\mathfrak{e}$ is a full extension.  Therefore, $\sigma_{ \mathfrak{e} } ( e_{n} )$ is Murray-von Neumann equivalent to $1_{ \multialg{ \mathfrak{A}[2] } }$.  Hence, $e_{n} \mathfrak{A}[2] e_{n} \cong \mathfrak{A}[2] \cong \K$.  Set $\mathfrak{A}_{n} = e_{n} \mathfrak{A} e_{n}$ and define $\ftn{ \beta_{n} }{ \mathfrak{A}_{n} }{ \mathfrak{A}_{n} }$ by $\beta_{n} ( x ) =\mathrm{Ad} ( v_{n} ) \circ \beta (x)$.  Then $\beta_{n}$ is a unital, full $X_{2}$-equivariant homomorphism.  Since  $\eta_{ \mathfrak{A}_{n} } \circ \sigma_{ \mathfrak{e}_{n} } \circ \beta_{n}$ is a unital representation of $\mathfrak{A}_{n}$, the closed subspace of $\ell^{2}$ generated by $\setof{ ( \eta_{ \mathfrak{A}_{n} } \circ \sigma_{ \mathfrak{e}_{n} } \circ \beta_{n} )(x) \xi }{ x \in \mathfrak{A}_{n} , \xi \in \ell^{2} }$ is $\ell^{2}$.  Therefore, $\eta_{\mathfrak{A}_{n} } \circ \sigma_{ \mathfrak{e}_{n} } \circ \beta_{n}$ is a non-degenerate representation of $\mathfrak{A}_{n }$.  

Since $K_{X_{2}} ( \beta ) = K_{X_{2}} ( \id_{ \mathfrak{A} } )$ and the $X_{2}$-equivariant embedding of $\mathfrak{A}_{n}$ as a sub-algebra of $\mathfrak{A}$ induces an isomorphism in ideal related $K$-theory, we have that $K_{X_{2}} ( \beta_{n} ) = K_{X_{2}} ( \id_{ \mathfrak{A}_{n} }  )$.  By Lemma~\ref{l:rep}, there exists a sequence of unitaries $W_{k,n} \in \multialg{ \mathfrak{A}_{n}[2] }$ such that 
\begin{align*}
( \mathrm{Ad} ( W_{k,n}  ) \circ \sigma_{ \mathfrak{e}_{n} } \circ \beta_{n} )(x) - \sigma_{ \mathfrak{e}_{n} } ( x ) \in \mathfrak{A}_{n}[2]
\end{align*}  
for all $x \in \mathfrak{A}_{n}$ and for all $k \in \N$, and
\begin{align*}
\lim_{ k \to \infty }  \| ( \mathrm{Ad} ( W_{k,n}  ) \circ \sigma_{ \mathfrak{e}_{n} } \circ \beta_{n} )(x) - \sigma_{ \mathfrak{e}_{n} } ( x ) \| = 0
\end{align*}  
for all $x \in \mathfrak{A}_{n}$.

Note that $\multialg{ \mathfrak{A}_{n}[2] } \cong \sigma_{ \mathfrak{e} } ( e_{n}) \multialg{ \mathfrak{A}[2] } \sigma_{ \mathfrak{e} } ( e_{n})$ with an isomorphism mapping $ \mathfrak{A}_{n}[2]$ onto $e_{n} \mathfrak{A}[2] e_{n}$.  Thus, we get a partial isometry $\widetilde{W}_{k,n}$ in $ \multialg{ \mathfrak{A}[2] }$ such that $\widetilde{W}_{k,n}^{*} \widetilde{W}_{k,n} = \widetilde{W}_{k,n} \widetilde{W}_{k,n}^{*} = \sigma_{ \mathfrak{e} } ( e_{n})$ and 
\begin{align*}
( \mathrm{Ad} ( \widetilde{W}_{k,n}  ) \circ \sigma_{ \mathfrak{e} } \circ \mathrm{Ad}(v_{n}) \circ \beta )(x) - \sigma_{ \mathfrak{e} } ( x ) \in \mathfrak{A}[2]
\end{align*}  
for all $x \in \mathfrak{A}_{n}$ and for all $k \in \N$, and
\begin{align*}
\lim_{ k \to \infty }  
\| ( \mathrm{Ad} ( \widetilde{W}_{k,n}  ) \circ \sigma_{ \mathfrak{e} } \circ \mathrm{Ad}(v_{n}) \circ \beta )(x) - \sigma_{ \mathfrak{e} } ( x ) \| = 0
\end{align*}  
for all $x \in \mathfrak{A}_{n}$.  

Set $V_{k,n} = ( \widetilde{W}_{k,n} + 1_{ \multialg{ \mathfrak{A}[2] } } - \sigma_{ \mathfrak{e} } ( e_{n} ) ) \sigma_{\mathfrak{e}} ( v_{n} )$.  Then $V_{k,n}$ is a unitary in $\multialg{ \mathfrak{A}[2]}$ such that 
\begin{align*}
( \mathrm{Ad} ( V_{k,n}  ) \circ \sigma_{ \mathfrak{e} } \circ \beta )(x) - \sigma_{ \mathfrak{e} } ( x ) \in \mathfrak{A}[2]
\end{align*}  
for all $x \in e_{n} \mathfrak{A} e_{n}$ and for all $k \in \N$, and
\begin{align*}
\lim_{ k \to \infty }  
\| ( \mathrm{Ad} ( V_{k,n}  ) \circ \sigma_{ \mathfrak{e} } \circ \beta )(x) - \sigma_{ \mathfrak{e} } ( x ) \| = 0
\end{align*}  
for all $x \in e_{n} \mathfrak{A} e_{n}$.  A consequence of the first part is that $( \mathrm{Ad} ( V_{k,n}  ) \circ \sigma_{ \mathfrak{e} } \circ \beta )(x) \in \sigma_{ \mathfrak{e} } ( e_{n} \mathfrak{A} e_{n} ) + \mathfrak{A}[2]$  for all $x \in e_{n} \mathfrak{A}e_{n}$.  Since $\beta_{\{1\}}  = \id_{ \mathfrak{A}[2] }$, we have that  $x - \beta( x ) \in \mathfrak{A}[2]$ for all $x \in e_{n}\mathfrak{A} e_{n}$.  Therefore,
\begin{align*}
\mathrm{Ad} ( V_{k,n} ) ( \sigma_{\mathfrak{e}} (x) ) = \mathrm{Ad} ( V_{k,n} ) \circ \sigma_{ \mathfrak{e} } ( x - \beta(x) ) + \mathrm{Ad} ( V_{k,n} ) \circ \beta (x) \in  \sigma_{ \mathfrak{e} } ( e_{n} \mathfrak{A} e_{n} ) + \mathfrak{A}[2]
\end{align*}  
Thus, $\alpha_{k,n} = \sigma_{\mathfrak{e} }^{-1} \circ ( \mathrm{Ad} ( V_{k,n}  ) \circ \sigma_{ \mathfrak{e} } \circ \mathrm{Ad}(v_{n}) ) \vert_{  e_{n} \mathfrak{A} e_{n} }$ is a homomorphism from $e_{n} \mathfrak{A} e_{n}$ to $\mathfrak{A}$. 

Since
\begin{align*}
\lim_{ k \to \infty }  
\| ( \mathrm{Ad} ( V_{k,n}  ) \circ \sigma_{ \mathfrak{e} } \circ \beta )(x) - \sigma_{ \mathfrak{e} } ( x ) \| = 0
\end{align*}  
for all $x \in e_{n} \mathfrak{A} e_{n}$ and $e_{n} \mathfrak{A} e_{n} \subseteq e_{n+1} \mathfrak{A} e_{n+1}$, there exists a strictly increasing sequence $\{ k(n) \}_{ n = 1}^{\infty}$ of positive integers such that 
\begin{align*}
\lim_{ n \to \infty } \| \alpha_{ k(n), n } \circ \beta(x) - x \| = 0
\end{align*}
for all $x \in \bigcup_{ n = 1}^{ \infty } e_{ n } \mathfrak{A} e_{n}$.  Let $\alpha_{n}$ be a completely, contractive, positive linear extension of $\alpha_{ k(n) , n }$.  Since $\bigcup_{ n = 1}^{ \infty } e_{ n } \mathfrak{A} e_{n}$ is dense in $\mathfrak{A}$, we have that 
\begin{align*}
\lim_{ n \to \infty } \| \alpha_{ n } \circ \beta(x) - x \| = 0
\end{align*}
for all $x \in  \mathfrak{A}$.  We have just proved the first part of the lemma. 

We now show that $\alpha_{n}$ can be chosen to be a homomorphism provide that $\mathfrak{A}$ is weakly semiprojective.  Suppose $\mathfrak{A}$ is weakly semiprojective.  Let $\epsilon > 0$ and $\mathcal{F}$ be a finite subset of $\mathfrak{A}$.  By Theorem~2.4 of \cite{sepBDF} (see also Definition~2.1 and Theorem~2.3 of \cite{hl:weaklysemi}, and Theorem~19.1.3 of \cite{loringlift}), there exist a $\delta > 0$ and a finite subset $\mathcal{G}$ of $\mathfrak{A}$ such that for any $C^{*}$-algebra $\mathfrak{B}$ and any contractive, completely positive, linear map $\ftn{L}{\mathfrak{A} }{ \mathfrak{B} }$ such that 

\begin{align*}
\| L( a b) - L(a)L(b) \| < \delta
\end{align*}
for all $a , b \in \mathcal{G}$, there exists a homomorphism $\ftn{ h }{ \mathfrak{A} }{ \mathfrak{B} }$ such that 
\begin{align*}
\| h(x) - L(x) \| < \frac{ \epsilon }{ 2 }
\end{align*}
for all $x \in \beta ( \mathcal{F} )$.

Without loss of generality, we may assume that $\epsilon < 1$ and $\delta < 1$.  Set 
\begin{align*}
M &= 1+\max \left(\setof{ \| a \| }{ a\in \mathcal{G}}\cup\setof{ \| x \|}{ x \in \mathcal{F} }\right)
\end{align*}
Since $e_{n} \mathfrak{A} e_{n} \subseteq e_{n+1} \mathfrak{A} e_{n+1}$ and $\bigcup_{ n = 1}^{ \infty  } e_{n} \mathfrak{A} e_{n}$ is dense in $\mathfrak{A}$, there exist $n \in \N$ and a finite subset $\mathcal{H} \subseteq e_{n} \mathfrak{A} e_{n}$ such that for each $a \in \mathcal{G}$, there exists $y \in \mathcal{H}$ such that $\| a - y \| < \frac{ \delta }{ 4M}$ and 
\begin{align*}
\| \alpha_{n} \circ \beta(x) - x \| < \frac{ \epsilon }{ 2 }
\end{align*}
for all $x \in \mathcal{F}$.
Let $a, b \in \mathcal{G}$.  Choose $x, y \in \mathcal{H} \subseteq e_{n} \mathfrak{A} e_{n}$ such that $\| a - x \| < \frac{ \delta }{ 4M }$ and $\| b - y \| < \frac{ \delta }{ 4M }$.  Note that $\| x\| \leq 1 + \| a \| \leq M$ and $\| y \| \leq 1 + \| b \| \leq M$.  Then 
\begin{align*}
\| \alpha_{n} ( ab) - \alpha_{n}(a) \alpha_{n} (b) \| &= \| \alpha_{n} ( ab - xb + xb - xy ) + \alpha_{n}(x y) - \alpha_{n} ( a) \alpha_{n} (b) \| \\
&\leq \| b \| \| a - x \| + \| x \| \| b - y \|  \\
&\qquad + \| \alpha_{n}(x) \alpha_{n}(y) - \alpha_{n}(x) \alpha_{n}(b) \| \\
&\qquad\quad+ \| \alpha_{n} (x) \alpha_{n}(b) - \alpha_{n}(a) \alpha_{n}(b) \| \\
&\leq 2M \| a - x \| + 2M \| b - y \| \\
&< 4M \frac{ \delta }{ 4M } = \delta.
\end{align*}
By the choice of $\delta$ and $\mathcal{G}$, there exists a homomorphism $\ftn{ \psi }{ \mathfrak{A} }{ \mathfrak{A} }$ such that 
\begin{align*}
\| \psi ( t ) - \alpha_{n}(t) \| < \frac{ \epsilon }{ 2 }
\end{align*}    
for all $t \in \beta ( \mathcal{F} )$.  Let $x \in \mathcal{F}$.  Then 
\begin{align*}
\| \psi \circ \beta(x) - x \| \leq \| \psi ( \beta(x) )  -  \alpha_{n} ( \beta (x) ) \| + \| \alpha_{n} ( \beta(x) ) - x \| < \frac{ \epsilon }{ 2 } + \frac{ \epsilon }{ 2 } = \epsilon.  
\end{align*}

We have just shown that for every $\epsilon > 0$ and for every finite subset $\mathcal{F}$ of $\mathfrak{A}$, there exists a homomorphism $\ftn{ \psi }{ \mathfrak{A} }{ \mathfrak{A} }$ such that 
\begin{align*}
\| \psi \circ \beta(x) - x \| < \epsilon
\end{align*}
for all $x \in \mathcal{F}$.  Consequently, there exists a sequence of endomorphisms $\{ \ftn{\psi_{n} }{ \mathfrak{A} }{ \mathfrak{A} } \}_{ n = 1}^{\infty}$ such that
\begin{align*}
\lim_{ n \to \infty } \| \psi_{n} \circ \beta(x) - x \| = 0
\end{align*}
for all $x \in \mathfrak{A}$ since $\mathfrak{A}$ is separable.
\end{proof}

To prove a uniqueness theorem involving tight $C^{*}$-algebras $\mathfrak{A}$ over $X_{2}$, we require that $\mathfrak{A}[1]$ belongs to a class of $C^{*}$-algebras whose injective homomorphisms between two objects in this class are classified by $\kk$.   

\begin{defin}\label{d:approxunitkk}
We will be interested in classes $\mathcal{C}$ of separable, nuclear, simple $C^{*}$-algebras  satisfying the following property that if $\mathfrak{A} , \mathfrak{B} \in \mathcal{C}$ and $\ftn{ \phi , \psi }{ \mathfrak{A} \otimes \K }{ \mathfrak{B} \otimes \K }$ are two injective homomorphisms such that $\kk ( \phi ) = \kk ( \psi )$, then $\phi$ and $\psi$ are approximately unitarily equivalent.
\end{defin}

\begin{remar}
$ $
\begin{itemize}
\item[(1)] By Theorem~4.1.3 of \cite{phillipspureinf} if $\mathcal{C}$ is the class of Kirchberg algebras, then $\mathcal{C}$ satisfies the property in Definition~\ref{d:approxunitkk}.

\item[(2)]  Let $\mathcal{C}$ be the class of unital, separable, nuclear, simple tracially AF $C^{*}$-algebras in $\mathcal{N}$.  Then $\mathcal{C}$ satisfies the property in Definition~\ref{d:approxunitkk}. 
\end{itemize}
\end{remar}

\begin{theor}(Uniqueness Theorem~2) \label{t:uniq2}
Let $\mathcal{C}$ be a class of $C^{*}$-algebras satisfying the property in Definition~\ref{d:approxunitkk} and let $\mathfrak{A}$ be a unital, separable, nuclear, tight $C^{*}$-algebra over $X_{2}$ such that $\mathfrak{A}[2] \cong \K$ and $\mathfrak{A}[1] \in \mathcal{C}$.  Suppose $\mathfrak{A} \otimes \K$ is semiprojective and $\mathfrak{A}$ has the stable weak cancellation property.  Let $\ftn{ \phi }{ \mathfrak{A} \otimes \K }{ \mathfrak{A} \otimes \K }$ be a full $X_{2}$-equivariant homomorphism such that $\kk ( X_{2} ; \phi ) = \kk (  X_{2} ; \id_{ \mathfrak{A} \otimes \K} )$.  Then there exists a sequence of full $X_{2}$-equivariant endomorphisms $\{ \ftn{ \alpha_{n} }{ \mathfrak{A} \otimes \K }{ \mathfrak{A} \otimes \K } \}_{ n = 1}^{\infty}$ such that $\kk (  X_{2} ; \alpha_{n} ) = \kk ( X_{2} ; \id_{ \mathfrak{A} \otimes \K } )$ and 
\begin{align*}
\lim_{ n \to \infty } \| (\alpha_{n} \circ \phi)(x) - x \| = 0
\end{align*}  
for all $x \in \mathfrak{A} \otimes \K$.
\end{theor}

\begin{proof}
Set $\mathfrak{B} = \mathfrak{A} \otimes \K$.  Note that $\mathfrak{B}$ is a tight $C^{*}$-algebra over $X_{2}$ with $\mathfrak{B} [ 2 ] \cong \K$.  Throughout the proof, $\ftn{\pi}{ \mathfrak{B} }{ \mathfrak{B} [1 ]}$ will denote the canonical projection.  Note that $\kk ( \phi_{ \{1\} } ) = \kk ( \id_{ \mathfrak{B} [1] } )$ since $\kk ( X_{2} ; \phi ) = \kk ( X_{2} ; \id_{ \mathfrak{B} } )$.   Since $\mathfrak{A}[1] \in \mathcal{C}$, there exists a sequence of unitaries $\{ z_{k} \}_{ k =1}^{ \infty }$ in $\multialg{ \mathfrak{B}[1] }$ such that
\begin{align*}
\lim_{ k \to \infty } \| z_{k}  \phi_{\{1\}} ( \pi(b) )  z_{k}^{*} - \pi (b) \| = 0
\end{align*}
for all $b \in \mathfrak{B}$.  Using the fact that $\phi$ is an $X_{2}$-equivariant homomorphism, we have that $\pi \circ \phi  = \phi_{ \{1\} } \circ \pi$, and hence
\begin{align*}
\lim_{ k \to \infty } \| z_{k} ( \pi \circ \phi(b) )  z_{k}^{*} - \pi (b) \| = 0
\end{align*}
for all $b \in \mathfrak{B}$.

Let $\ftn{ \overline{ \pi } }{ \multialg{ \mathfrak{B} } }{ \multialg{ \mathfrak{B} [1 ] } }$ be the surjective homomorphism induced by $\pi$.  Since $\mathfrak{B}$ is stable, by Corollary~2.3 of \cite{RordamStable}, we have that $\mathfrak{B} [1]$ is stable.  Thus, the unitary group of $\multialg{ \mathfrak{B} [ 1 ] }$ is path-connected, which implies that every unitary in $\multialg{ \mathfrak{B}[1] }$ lifts to a unitary in $\multialg{ \mathfrak{B} }$.  Hence, there exists a sequence of unitaries $\{w_{k} \}_{ k = 1}^{ \infty }$ in $\multialg{ \mathfrak{B} }$ such that $\overline{ \pi } ( w_{k} ) = z_{k}$.  Since $\mathfrak{B}$ is semiprojective, by Proposition~2.2 of \cite{mdge_onepara} (see \cite{loringlift}), there exists a sequence of homomorphisms $\{ \ftn{ \beta_{ \ell } }{ \mathfrak{B} }{ \mathfrak{B} } \}_{ \ell = 1}^{ \infty }$ and a strictly increasing sequence $\{ k(\ell) \}_{ \ell =1}^{ \infty }$ of positive integers such that $\pi \circ \beta_{ \ell } = \pi$ and 
\begin{align*}
\lim_{ \ell \to \infty } \| \mathrm{Ad} ( w_{ k(\ell) } ) \circ \phi ( b) - \beta_{ \ell } ( b ) \| = 0
\end{align*} 
for all $b \in \mathfrak{B}$.   

By Remark~\ref{r:X2hom}, there exists $N_{1} \in \N$ such that $\beta_{\ell}$ is a full $X_{2}$-equivariant homomorphism for all $\ell \geq N_{1}$.   By Proposition~2.3 of \cite{mdge_onepara}, we may choose $N_{2} \geq N_{1}$ such that for all $\ell \geq N_{2}$, we have that $\beta_{\ell}$ and $\mathrm{Ad} ( w_{ k(\ell) } ) \circ \phi$ is homotopic.  It follows from Theorem~5.5 of \cite{mdrm:ethy} that $\kk ( X_{2} ; \beta_{\ell} ) = \kk ( X_{2} ; \mathrm{Ad} ( w_{ k(\ell) } ) \circ \phi ) = \kk ( X_{2} ; \phi ) = \kk ( X_{2} ; \id_{ \mathfrak{B} } )$.

Let $\ell \geq N_{2}$.  Note that $( \beta_{\ell} )_{\{1\}} = \id_{ \mathfrak{B}[1] }$ since $\pi \circ \beta_{ \ell } = \pi$.  Since $\mathfrak{A}$ is semiprojective, by Corollary~3.6 of \cite{md:contfields} (also see Chapter~19 of \cite{loringlift}), $\mathfrak{A}$ is weakly semiprojective.  Hence, by Lemma~\ref{l:rep2}, there exists a sequence of homomorphisms $\{ \ftn{ \alpha_{ m, \ell } }{ \mathfrak{B} }{ \mathfrak{B} } \}_{ m = 1}^{\infty }$ such that 
\begin{align*}
\lim_{ m \to \infty } \| \alpha_{m,\ell } \circ \beta_{\ell} (x) -x \| = 0
\end{align*}
for all $x \in \mathfrak{B}$.  Since $\beta_{\ell}$ and $\id_{\mathfrak{B}}$ are full $X_{2}$-equivariant homomorphisms, by Remark~\ref{r:X2hom}, there exists $N_{3}$ such that, for all $m \geq N_{3}$, we have that $\alpha_{m,\ell}$ is a full $X_{2}$-equivariant homomorphism.  Moreover, by Proposition~2.3 of \cite{mdge_onepara}, we can choose $N_{3} \geq N_{2}$ such that $\alpha_{m,\ell} \circ \beta_{\ell}$ and $\id_{ \mathfrak{B} }$ are homotopic.  It follows from Theorem~5.5 of \cite{mdrm:ethy} that $\kk( X_{2} ; \alpha_{  m , \ell} \circ \beta_{ \ell } ) = \kk ( X_{2} ; \id_{ \mathfrak{B} } )$ for all $m \geq N_{3}$.  Consequently, $\kk ( X_{2} ; \alpha_{ m , \ell} ) =  \kk ( X_{2} ; \id_{ \mathfrak{B} } )$ for all $m \geq N_{3}$ since $\kk( X_{2} ; \beta_{ \ell } ) = \kk ( X_{2} ; \id_{ \mathfrak{B} } )$.

Let $\mathcal{F}$ be a finite subset of $\mathfrak{B}$ and $\epsilon > 0$.  Then there exists $\ell \geq N_{2}$ such that 
\begin{align*}
\| \mathrm{Ad} ( w_{ k(\ell) } ) \circ \phi ( b ) - \beta_{ \ell } ( b ) \| < \frac{ \epsilon }{ 2 }
\end{align*}  
for all $b \in \mathcal{F}$.  Moreover, there exists $m \geq N_{3}$ such that 
\begin{align*}
\| \alpha_{  m , \ell} \circ  \beta_{\ell} (b) -  b \| < \frac{ \epsilon }{ 2 }
\end{align*}
for all $b \in \mathcal{F}$.  Set $\alpha_{1} = \mathrm{Ad} ( w_{k( \ell ) } ) \vert_{ \mathfrak{B} }$ and $\alpha = \alpha_{ m , \ell } \circ \alpha_{1}$.  Since $w_{k(\ell)}$ is a unitary in $\multialg{ \mathfrak{B} }$, we have that $\alpha_{1}$ is an automorphism of $\mathfrak{B}$ and $\kk ( X_{2} ; \alpha_{1} ) = \kk ( X_{2} ; \id_{ \mathfrak{B} } )$.  Therefore, $\alpha$ is a full $X_{2}$-equivariant homomorphism.  Since $\ell \geq N_{2}$ and $m \geq N_{3}$, we have that $\kk ( X_{2} ; \alpha_{ m , \ell } ) = \kk ( X_{2} ; \id_{ \mathfrak{B} } )$.  Therefore, $\kk ( X_{2} ; \alpha ) = \kk ( X_{2} ; \id_{ \mathfrak{B} } )$.  Let $b \in \mathcal{F}$.  Then 
\begin{align*}
\| \alpha \circ \phi ( b ) - b \| &= \| \alpha_{ m, \ell } \circ \mathrm{Ad}( w_{k(\ell) } )  \circ \phi ( b ) - b \| \\
&\leq  \| \alpha_{ m , \ell } \circ \mathrm{Ad} ( w_{k(\ell) } ) \circ \phi ( b ) - \alpha_{ m , \ell } \circ \beta_{ \ell } ( b ) \| + \| \alpha_{ m , \ell }  \circ  \beta_{\ell} (b) -  b \| \\
						&< \frac{ \epsilon }{ 2 } + \frac{ \epsilon }{ 2 }  = \epsilon.
\end{align*}

We have just shown that for every $\epsilon > 0$ and for every finite subset $\mathcal{F}$ of $\mathfrak{B}$, there exists a full $X_{2}$-equivariant homomorphism $\ftn{ \alpha }{ \mathfrak{B} }{ \mathfrak{B} }$ such that $\kk ( X_{2} ; \alpha) = \kk ( X_{2} ; \id_{ \mathfrak{B} } )$ and 
\begin{align*}
\| \alpha \circ \phi ( b ) - b \| < \epsilon 
\end{align*}
for all $b \in \mathfrak{B}$.  Since $\mathfrak{B}$ is a separable $C^{*}$-algebra, there exists a sequence of full $X_{2}$-equivariant homomorphisms $\{ \ftn{ \alpha_{n} }{ \mathfrak{B} }{ \mathfrak{B} } \}_{ n =1}^{ \infty }$ such that $\kk ( X_{2} ; \alpha_{n} ) = \kk (  X_{2} ; \id_{ \mathfrak{B} } )$ and 
\begin{align*}
\lim_{ n \to \infty } \| \alpha_{n} \circ \phi ( b ) - b \| = 0
\end{align*}
for all $b \in \mathfrak{B}$.
\end{proof}

\begin{theor}\label{t:classmixed2}
Let $\mathcal{C}$ be a class of $C^{*}$-algebras satisfying the property in Definition~\ref{d:approxunitkk}, and let $\mathfrak{A}_{1}$ and $\mathfrak{A}_{2}$ be unital, separable, nuclear, tight $C^{*}$-algebras over $X_{2}$ such that $\mathfrak{A}_{i} [ 2 ] \cong \K$ and $\mathfrak{A}_{i} [ 1 ] \in \mathcal{C}$.  Suppose $\mathfrak{A}_{i} \otimes \K$ is semiprojective and $\mathfrak{A}_{i}$ has the stable weak cancellation property.  If there exist full $X_{2}$-equivariant homomorphisms, $\ftn{ \phi }{ \mathfrak{A}_{1} \otimes \K }{ \mathfrak{A}_{2} \otimes \K }$ and $\ftn{ \psi }{ \mathfrak{A}_{2} \otimes \K }{ \mathfrak{A}_{1} \otimes \K }$, such that $\kk ( X_{2} ; \phi \circ \psi) = \kk ( X_{2} ; \id_{ \mathfrak{A}_{2} \otimes \K } )$ and $\kk ( X_{2} ; \psi \circ \phi) = \kk ( X_{2} ; \id_{ \mathfrak{A}_{1} \otimes \K } )$, then for any finite subset $\mathcal{F}$ and $\epsilon > 0$, there exists an isomorphism $\ftn{\gamma }{ \mathfrak{A}_{1} \otimes \K }{ \mathfrak{A}_{2} \otimes \K }$ such that $\kk( X_{2} ; \gamma ) = \kk ( \phi )$ and
\begin{align*}
\| \gamma (x) - \phi (x) \| < \epsilon
\end{align*}
for all $x \in \mathcal{F}$.
\end{theor}

\begin{proof}
Let $\{ \overline{\mathcal{F}}_{n} \}_{ n = 1}^{\infty}$ be a sequence of finite subsets of $\mathfrak{A}_{1} \otimes \K $ such that $\overline{\mathcal{F}}_{n} \subseteq \overline{\mathcal{F}}_{n+1}$ and $\bigcup_{ n = 1}^{\infty } \overline{\mathcal{F}}_{n}$ is dense in $\mathfrak{A}_{1} \otimes \K$ and let $\{ \overline{\mathcal{G}}_{n} \}_{ n = 1}^{\infty}$ be a sequence of finite subsets of $\mathfrak{A}_{2} \otimes \K $ such that $\overline{\mathcal{G}}_{n} \subseteq \overline{\mathcal{G}}_{n+1}$ and $\bigcup_{ n = 1}^{\infty } \overline{\mathcal{G}}_{n}$ is dense in $\mathfrak{A}_{2} \otimes \K$.

Let $\epsilon > 0$ and $\mathcal{F}$ be a finite subset of $\mathfrak{A}_{1}$.  Set $\mathcal{F}_{1} = \mathcal{F} \cup \overline{\mathcal{F}}_{1}$ and choose $m_{1} \in \N$ such that $\sum_{ k = m_{1} }^{\infty} \frac{1}{ 2^{k} } < \epsilon$.  By Theorem~\ref{t:uniq2}, there exists a full $X_{2}$-equivariant homomorphism $\ftn{ \alpha_{1} }{ \mathfrak{A}_{1} \otimes \K }{ \mathfrak{A}_{1} \otimes \K }$ such that $\kk ( X_{2} ; \alpha_{1}  ) = \kk ( X_{2} ; \id_{ \mathfrak{A}_{1} \otimes \K } )$ and 
\begin{align*}
\| \alpha_{1} \circ \psi \circ \phi ( a ) - a \| < \frac{1}{2^{m_{1} + 1}}
\end{align*}
for all $a \in \mathcal{F}_{1}$.  Set $\phi_{1} = \phi$ and $\psi_{1} = \alpha_{1} \circ \psi$.  Then $\kk ( X_{2} ; \psi_{1} ) = \kk ( X_{2} ;  \psi )$ and $\| \psi_{1} \circ \phi_{1}(a) - a \| < \frac{1}{2^{m_{1} +1}}$ for all $a \in \mathcal{F}_{1}$.   

Set $\mathcal{G}_{1} = \overline{G}_{1} \cup \phi( \mathcal{F}_{1} )$.  Note that $\kk ( X_{2} ; \phi \circ \psi_{1} ) =\kk ( X_{2} ; \phi \circ \psi ) = \kk ( X_{2} ; \id_{ \mathfrak{A}_{2} \otimes \K } )$.  Hence, by Theorem~\ref{t:uniq2}, there exists a full $X_{2}$-equivariant homomorphism $\ftn{ \beta_{1} }{ \mathfrak{A}_{2} \otimes \K }{ \mathfrak{A}_{2} \otimes \K }$ such that $\kk ( X_{2} ; \beta_{1} ) = \kk ( X_{2} ; \id_{ \mathfrak{A}_{2} \otimes \K } )$ and 
\begin{align*}
\| \beta_{1} \circ \phi \circ \psi_{1}( x ) - x \| < \frac{1}{2^{m_{1} +1}}
\end{align*}
for all $x \in \mathcal{G}_{1}$.  Set $\phi_{2} = \beta_{1} \circ \phi$.  Then $\kk ( X_{2} ; \phi_{2} ) = \kk ( X_{2} ; \phi )$ and 
\begin{align*}
\| \phi_{2} \circ \psi_{1} ( x ) - x \| < \frac{1}{2^{m_{1}+1}}
\end{align*}
for all $x \in \mathcal{G}_{1}$.  Note that for all $x \in \mathcal{F}_{1}$, then 
\begin{align*}  
\| \phi(x) - \phi_{2} (x) \| &\leq \| \phi_{1} (x)  - \phi_{2} \circ \psi_{1} ( \phi_{1} (x) ) \| + \| \phi_{2} \circ \psi_{1} ( \phi_{1} (x) ) - \phi_{2} ( x ) \| \\
			& < \frac{1}{2^{m_{1}+1}} + \| \psi_{1} \circ\phi_{1} (x)  - x \| < \frac{1}{2^{m_{1}}}.
\end{align*}

Set $\mathcal{F}_{2} = \overline{ \mathcal{F} }_{2} \cup \phi_{2} ( \mathcal{G}_{1} )$.  Note that $\kk ( X_{2} ; \psi \circ \phi_{2} ) =\kk ( X_{2} ; \psi \circ \phi  ) = \kk ( X_{2} ; \id_{ \mathfrak{A}_{1} \otimes \K } )$.  Hence, by Theorem~\ref{t:uniq2}, there exists a full $X_{2}$-equivariant homomorphism $\ftn{ \alpha_{2} }{ \mathfrak{A}_{1} \otimes \K }{ \mathfrak{A}_{1} \otimes \K }$ such that $\kk ( X_{2} ; \alpha_{2} ) = \kk ( X_{2} ; \id_{ \mathfrak{A}_{1} \otimes \K } )$ and 
\begin{align*}
\| \alpha_{2} \circ \psi \circ \phi_{2}( a ) - a \| < \frac{1}{2^{m_{1}+2}}
\end{align*}
for all $a \in \mathcal{F}_{2}$.  Set $\psi_{2} = \alpha_{2} \circ \psi$.  Then $\kk ( X_{2} ; \psi_{2}  ) = \kk ( X_{2} ; \psi)$ and 
\begin{align*}
\| \psi_{2} \circ \phi_{2} ( a ) - a \| < \frac{1}{2^{m_{1} +2 } }
\end{align*}
for all $x \in \mathcal{F}_{2}$.

Set $\mathcal{G}_{2} = \overline{G}_{2} \cup \phi_{2}( \mathcal{F}_{2} )$.  Note that $\kk ( X_{2} ; \phi \circ \psi_{2} ) =\kk ( X_{2} ; \phi \circ \psi ) = \kk ( X_{2} ; \id_{ \mathfrak{A}_{2} \otimes \K } )$.  Hence, by Theorem~\ref{t:uniq2}, there exists a full $X_{2}$-equivariant homomorphism $\ftn{ \beta_{2} }{ \mathfrak{A}_{2} \otimes \K }{ \mathfrak{A}_{2} \otimes \K }$ such that $\kk ( X_{2} ; \beta_{2} ) = \kk ( X_{2} ; \id_{ \mathfrak{A}_{2} \otimes \K } )$ and 
\begin{align*}
\| \beta_{2} \circ \phi \circ \psi_{2}( x ) - x \| < \frac{1}{2^{m_{1} +2}}
\end{align*}
for all $x \in \mathcal{G}_{2}$.  Set $\phi_{3} = \beta_{2} \circ \phi$.  Then $\kk ( X_{2} ; \phi_{3} ) = \kk ( X_{2} ; \phi )$ and 
\begin{align*}
\| \phi_{3} \circ \psi_{2} ( x ) - x \| < \frac{1}{2^{m_{1}+2}}
\end{align*}
for all $x \in \mathcal{G}_{2}$.  Note that for all $x \in \mathcal{F}_{2}$, we have that
\begin{align*}  
\| \phi_{2}(x) - \phi_{3} (x) \| &\leq \| \phi_{2} (x)  - \phi_{3} \circ \psi_{2} ( \phi_{2} (x) ) \| + \| \phi_{3} \circ \psi_{2} ( \phi_{2} (x) ) - \phi_{3} ( x ) \| \\
			& < \frac{1}{2^{m_{1}+2}} + \| \psi_{2} ( \phi_{2} (x) ) - x \| < \frac{1}{2^{m_{1}+1}}.
\end{align*}

Continuing this process, we have constructed a sequence $\{ \mathcal{F}_{n} \}_{ n = 1}^{\infty}$ of finite subsets of $\mathfrak{A}_{1} \otimes \K$, a sequence $\{ \mathcal{G}_{n} \}_{ n = 1}^{\infty}$ of finite subsets of $\mathfrak{A}_{2} \otimes \K$, a sequence of full $X_{2}$-equivariant homomorphisms $\{ \ftn{ \phi_{n} }{ \mathfrak{A}_{1} \otimes \K }{ \mathfrak{A}_{2} \otimes \K } \}_{ n = 1}^{\infty }$, and a sequence of full $X_{2}$-equivariant homomorphisms $\{ \ftn{ \psi_{n} }{ \mathfrak{A}_{2} \otimes \K }{ \mathfrak{A}_{1} \otimes \K } \}_{ n = 1}^{\infty }$ such that 
\begin{itemize}
\item[(1)] $\kk ( X_{2} ;  \phi_{n}  ) = \kk ( X_{2} ; \phi )$ for all $n \in \N$ and $\phi_{1} = \phi$;

\item[(2)] $\kk ( X_{2} ;  \psi_{n}  ) = \kk ( X_{2} ; \psi )$ for all $n \in \N$; 

\item[(3)] $\mathcal{F}_{n} \subseteq \mathcal{F}_{n+1}$ and $\overline{ \mathcal{F} }_{n} \subseteq \mathcal{F}_{n}$;

\item[(4)] $\mathcal{G}_{n} \subseteq \mathcal{G}_{n+1}$ and $\overline{ \mathcal{G} }_{n} \subseteq \mathcal{G}_{n}$;

\item[(5)] for each $x \in \mathcal{F}_{n}$ and for each $x \in \mathcal{G}_{n}$
\begin{align*}
 \| \psi_{n} \circ \phi_{n} (x) - x \| < \frac{1}{2^{m_{1}+n}}  \quad \text{and} \quad  \| \phi_{n+1} \circ \psi_{n} ( x ) - x \|  <  \frac{1}{2^{m_{1}+n}}
\end{align*}

\item[(6)] for each $x \in \mathcal{F}_{n}$, 
\begin{align*}
\| \phi_{n} (x) - \phi_{n+1} (x) \| < \frac{ 1 }{ 2^{m_{1} + n-1 } } 
\end{align*}
\end{itemize}
Since $\bigcup_{ n = 1}^{\infty} \overline{\mathcal{F} }_{n}$ is dense in $\mathfrak{A}_{1} \otimes \K$ and $\overline{ \mathcal{F} }_{n} \subseteq \mathcal{F}_{n}$, we have that $\bigcup_{ n = 1}^{\infty} \mathcal{F}_{n}$ is dense in $\mathfrak{A}_{1} \otimes \K$.  Similarly, $\bigcup_{ n = 1}^{\infty} \mathcal{G}_{n}$ is dense in $\mathfrak{A}_{2} \otimes \K$.  Therefore, there exists an isomorphism $\ftn{ \gamma }{ \mathfrak{A}_{1} \otimes \K }{ \mathfrak{A}_{2} \otimes \K }$ such that 
\begin{align*}
\| \gamma ( a ) - \phi_{n} ( a )\| < \sum_{ k = m_{1} + n - 1 }^{ \infty } \frac{1}{2^{k} }
\end{align*}
for all $a \in \mathcal{F}_{n}$.  Since $\mathcal{F} \subseteq \mathcal{F}_{1}$, we have that 
\begin{align*}
\| \phi (x) - \gamma (x) \| = \| \phi_{1} (x) - \gamma (x) \| < \sum_{ k = m_{1}}^{\infty} \frac{1}{2^{k} } < \epsilon.
\end{align*}
Since 
\begin{align*}
\lim_{ n \to \infty }\sum_{ k = m_{1} + n - 1 }^{ \infty } \frac{1}{2^{k} } = 0,
\end{align*}
we have that
\begin{align*}
\lim_{ n \to \infty } \| \gamma ( a) - \phi_{n} (a) \| = 0
\end{align*}  
for all $a \in \mathfrak{A}_{1} \otimes \K$.  Since $\mathfrak{A}_{1} \otimes \K$ is semiprojective, by Proposition~2.3 of \cite{mdge_onepara}, there exists $N \in \N$ such that $\gamma$ and $\phi_{N}$ are homotopic.  Hence, by Theorem~5.5 of \cite{mdrm:ethy}, $\kk ( X_{2} ; \gamma ) = \kk ( X_{2} ; \phi_{N} ) = x$.
\end{proof}

\subsection{Unital Classification}

We know combine the above results with the Meta-theorem of Section~\ref{s:meta} (see Theorem~\ref{t:metathm}) to get a strong classification for a class of unital $C^{*}$-algebras which includes all unital graph $C^{*}$-algebras with exactly one non-trivial ideal.

\begin{corol}\label{c:unitclassmixed1}
Let $\mathfrak{A}_{1}$ and $\mathfrak{A}_{2}$ be unital, tight $C^{*}$-algebras over $X_{n}$ such that $\mathfrak{A}_{i}$ has real rank zero, $\mathfrak{A}_{i} [n]$ is a Kirchberg algebra in $\mathcal{N}$, and $\mathfrak{A}_{i} [1,n-1]$ is an AF-algebra. Let $x \in \kk ( X_{2} ; \mathfrak{A}_{1} , \mathfrak{A}_{2} )$ be an invertible such that $K_{X_{n}}(x)_{ Y }$ is an order isomorphism for each $Y \in \mathbb{LC} ( X_{n} )$ and $K_{X_{n}}(x)_{X_{n}} ( [ 1_{ \mathfrak{A}_{1} } ] ) = [ 1_{ \mathfrak{A}_{2} } ]$ in $K_{0} ( \mathfrak{A}_{2} )$.  Then there exists an isomorphism $\ftn{ \phi }{ \mathfrak{A} }{ \mathfrak{B} }$ such that $K_{X_{n}} ( \phi ) = K_{X_{n}}(x)$. 
\end{corol}

\begin{proof}
Since $\mathfrak{A}_{i}[1]$ and $\mathfrak{A}_{i}[2]$ are separable and nuclear, we have that $\mathfrak{A}_{i}$ is separable and nuclear.  Since $\mathfrak{A}_{i} [ 1, n-1 ]$ is an AF-algebra and $\mathfrak{A}_{i}[n]$ is a Kirchberg algebra, they both have the stable weak cancellation property.  By Lemma~3.15 of \cite{err:fullext}, $\mathfrak{A}_{i}$ has stable weak cancellation property.  By Lemma~\ref{l:unitaryidealkthy}, for each tight $C^{*}$-algebra $\mathfrak{A}$ over $X_{n}$, we have that $K_{X_{n}} ( \mathrm{Ad} ( u ) \vert_{ \mathfrak{A} } )$ for each unitary $u \in \multialg{ \mathfrak{A} }$.  A computation shows that $K_{X_{n}} ( - )$ satisfies (1), (2), and (3) of Theorem~\ref{t:metathm} since $K_{*} ( - )$ does.  The corollary now follows from Theorem~\ref{t:metathm} and Theorem~\ref{t:classmixed1}.
\end{proof}

\begin{corol}\label{c:unitclassmixed2}
Let $\mathfrak{A}_{1}$ and $\mathfrak{A}_{2}$ be unital, tight $C^{*}$-algebras over $X_{2}$ such that $\mathfrak{A}_{i} [ 2 ] \cong \K$ and $\mathfrak{A}_{i} [ 1 ]$ is a Kirchberg algebra in $\mathcal{N}$.  Let $x \in \kk ( X_{2} ; \mathfrak{A}_{1} , \mathfrak{A}_{2} )$ be an invertible such that $K_{X_{2}}(x)_{ Y }$ is an order isomorphism for each $Y \in \mathbb{LC} ( X_{2} )$ and $K_{X_{2}}(x)_{X_{2}} ( [ 1_{ \mathfrak{A}_{1} } ] ) = [ 1_{ \mathfrak{A}_{2} } ]$ in $K_{0} ( \mathfrak{A}_{2} )$.  If $\mathfrak{A}_{i} \otimes \K$ is semiprojective, then there exists an isomorphism $\ftn{\gamma }{ \mathfrak{A}_{1} \otimes \K }{ \mathfrak{A}_{2} \otimes \K }$ such that $\kk( X_{2} ; \gamma ) = x$.  
\end{corol}

\begin{proof}
Since $\mathfrak{A}_{i}[1]$ and $\mathfrak{A}_{i}[2]$ are separable and nuclear, we have that $\mathfrak{A}_{i}$ is separable and nuclear.  Since $\mathfrak{A}_{i}[2]$ and $\mathfrak{A}_{i}[1]$ have real rank zero and $K_{1} ( \mathfrak{A}_{i}[2] ) = 0$, we have that $\mathfrak{A}$ has real rank zero.  Since $\mathfrak{A}_{i} [ 2 ]$ is an AF-algebra and $\mathfrak{A}_{i}[1]$ is a Kirchberg algebra, they both have the stable weak cancellation property.  Therefore, by Lemma~3.15 of \cite{err:fullext}, $\mathfrak{A}$ has the stable weak cancellation property.   

By Lemma~1.5 of \cite{ERRshift}, the extension $0 \to \mathfrak{A}_{i} [ 2 ] \to \mathfrak{A}_{i} \to \mathfrak{A}_{i}[1] \to 0$ is full, and hence by Proposition~1.6 of \cite{ERRshift}, $0 \to \mathfrak{A}_{i} [ 2 ]  \otimes \K \to \mathfrak{A}_{i} \otimes \K \to \mathfrak{A}_{i}[1]  \otimes \K \to 0$ is full.  The corollary now follows from Theorem~\ref{t:exist}(ii), Theorem~\ref{t:classmixed2}, and Theorem~\ref{t:metathm}.
\end{proof}

It is an open question to determine if every unital, separable, nuclear, tight $C^{*}$-algebra   $\mathfrak{A}$ over $X_{2}$ whose unique proper nontrivial ideal is isomorphic to $\K$ and quotient is a Kirchberg algebra in $\mathcal{N}$ with finitely generated $K$-theory is semiprojective.  The following results show that under some $K$-theoretical conditions, $\mathfrak{A}$ is semiprojective. 

\begin{lemma}\label{l:graphsemiprojective}
Let $E$ be a graph with finitely many vertices such that $C^{*} (E)$ is a tight $C^{*}$-algebra over $X_{2}$ with $C^{*} (E) [1]$ being purely infinite.  Then $C^{*} (E)$ and $C^{*} (E) \otimes \K$ are semiprojective. 
\end{lemma}

\begin{proof}
The fact that $C^{*} (E)$ is semiprojective follows from the results of \cite{setk:semiproj}.  By Proposition~6.4 of \cite{semt_classgraphalg}, $C^{*} (E) [ 2 ]$ is stable.  Since $C^{*} (E)$ is a unital $C^{*}$-algebra, by Lemma~1.5 of \cite{ERRshift}, the extension $\mathfrak{e} : 0 \to C^{*} (E) [2] \to C^{*} (E) \to C^{*} (E)[1] \to 0$ is a full extension.  By Proposition~3.21 and Corollary~3.22 of \cite{err:fullext}, $C^{*} (E)$ is properly infinite.  Therefore, by Theorem~4.1 of \cite{bb:semiproj}, $C^{*} (E) \otimes \K$ is semiprojective.    
\end{proof}

\begin{propo}\label{p:lookslikegraph}
Let $\mathfrak{A}$ be unital, separable, nuclear, tight $C^{*}$-algebras over $X_{2}$.  If $\mathfrak{A}[2] \cong \K$ and $\mathfrak{A}[1]$ is a Kirchberg algebra in $\mathcal{N}$ such that $\mathrm{rank} ( K_{1} ( \mathfrak{A}[1]  ) ) \leq \mathrm{rank} ( K_{0} ( \mathfrak{A} [1] ) )$, $K_{1} ( \mathfrak{A}[1] )$ is free, and the $K$-groups of $\mathfrak{A}[{i}]$ are finitely generated, then $\mathfrak{A}$ and $\mathfrak{A} \otimes \K$ are semiprojective.  Consequently, $\mathfrak{A}$ semiprojective.
\end{propo}

\begin{proof}
By Lemma~1.5 of \cite{ERRshift}, $\mathfrak{e} : 0 \to \mathfrak{A}[2] \to \mathfrak{A} \to \mathfrak{A}[1] \to 0$ is a full extension.  By Corollary~3.22 of \cite{err:fullext}, $K_{0} ( \mathfrak{A} )_{+} = K_{0} ( \mathfrak{A} )$.  By Theorem~6.4 of \cite{ektw:range}, there exists a graph $E$ with finitely many vertices such that $K_{X_{2}}^{+} ( \mathfrak{A} ) \cong K_{X_{2}}^{+} ( C^{*} (E) )$ such that $C^{*} (E)$ is a tight $C^{*}$-algebra over $X_{2}$.  Since $E$ has finitely many vertices, $C^{*} (E)$ is unital.  Since $K_{X_{2}}^{+} ( \mathfrak{A} ) \cong K_{X_{2}}^{+} ( C^{*} (E) )$, we have that $C^{*} (E)[1]$ is a Kirchberg algebra.  By Theorem~3.9 of \cite{ERRshift}, we have that $\mathfrak{A}\otimes \K \cong C^{*} (E) \otimes \K$.  By Lemma~\ref{l:graphsemiprojective}, $C^{*} (E)$ and $C^{*} (E) \otimes \K$ are semiprojective.  Hence, by Proposition~2.7 of \cite{bb:semiproj}, $\mathfrak{A}$ and $\mathfrak{A} \otimes \K$ are semiprojective. 
   \end{proof}
   
   \begin{corol}\label{c:classmixed2lookslikegraph}
Let $\mathfrak{A}_{1}$ and $\mathfrak{A}_{2}$ be unital, tight $C^{*}$-algebras over $X_{2}$ such that $\mathfrak{A}_{i} [ 2 ] \cong \K$ and $\mathfrak{A}_{i} [ 1 ]$ is a Kirchberg algebra in $\mathcal{N}$ such that $\mathrm{rank} ( K_{1} ( \mathfrak{A}[1]  ) ) \leq \mathrm{rank} ( K_{0} ( \mathfrak{A} [1] ) )$, $K_{1} ( \mathfrak{A}[1] )$ is free, and the $K$-groups of $\mathfrak{A}_{i}$ are finitely generated.  Let $x \in \kk ( X_{2} ; \mathfrak{A}_{1} , \mathfrak{A}_{2} )$ be an invertible such that $K_{X_{2}}(x)_{ Y }$ is an order isomorphism for each $Y \in \mathbb{LC} ( X_{2} )$ and $K_{X_{2}}(x)_{X_{2}} ( [ 1_{ \mathfrak{A}_{1} } ] ) = [ 1_{ \mathfrak{A}_{2} } ]$ in $K_{0} ( \mathfrak{A}_{2} )$.  Then there exists an isomorphism $\ftn{\gamma }{ \mathfrak{A}_{1} \otimes \K }{ \mathfrak{A}_{2} \otimes \K }$ such that $\kk( X_{2} ; \gamma ) = x$.  
\end{corol}

\begin{proof}
This follows from Proposition~\ref{p:lookslikegraph} and Corollary~\ref{c:unitclassmixed2}.
\end{proof}

\section{Applications}

Let $E$ be a graph satisfying Condition (K) (in particular, if $C^{*} (E)$ has finitely many ideals, then $E$ satisfies Condition (K)).  Let $\mathfrak{I}_{1}, \mathfrak{I}_{2}$ be ideals of $C^{*} ( E )$ such that $\mathfrak{I}_{1} \subseteq \mathfrak{I}_{2}$ and $\mathfrak{I}_{2} / \mathfrak{I}_{1}$ is simple.  Then by Theorem~5.1 of \cite{ermt:idealsgraph} and Corollary~3.5 of \cite{bgrs_idealstrucgraph}, $\mathfrak{I}_{2} / \mathfrak{I}_{1}$ is a simple graph $C^{*}$-algebra.  Hence, $\mathfrak{I}_{2} / \mathfrak{I}_{1}$ is either a Kirchberg algebra or an AF algebra.

\subsection{Classification of graph $C^{*}$-algebras with exactly one ideal}

\begin{lemma}\label{l:finitegraphaf}
Let $E$ be a graph with finitely many vertices such that $C^{*} (E)$ is a simple AF-algebra.  Then $C^{*} (E) \otimes \K \cong \K$.  Consequently, if $F$ is a graph with finitely many vertices such that $C^{*} (F)$ is a tight $C^{*}$-algebra over $X_{2}$ and $C^{*} (F)[2]$ is an AF-algebra, then $C^{*}(F)[2] \cong \K$.  
\end{lemma}

\begin{proof}
We claim that $E$ is a finite graph.  By Corollary~2.13 and Corollary~2.15 of \cite{DT1}, $E$ has no cycles, and for every vertex $v_{0}$ that emits infinitely many edges and for each vertex $v$, there exists a path from $v$ to $v_{0}$.  Since $E$ has no cycles, we have that every vertex of $E$ emits only finitely many edges.  Hence, $E$ is a finite graph.  By Proposition~1.18 of \cite{raeburn}, $C^{*} (E) \cong \mathsf{M}_{n}$.

We now prove the second statement.  First note that $C^{*} (F)[2]$ is a simple AF-algebra.  Since $C^{*} (F)[2]$ is stably isomorphic to a subgraph of $E$, $C^{*} (F)[2] \otimes \K \cong C^{*} (E)$ for some graph $E$ with finitely many vertices.  Since $C^{*} (E)$ is a simple AF-algebra, we have that $C^{*} (E) \otimes \K \cong \K$.   Hence, $C^{*} (F)[2] \otimes \K \cong \K$ which implies that $C^{*} (F)[2] \cong \mathsf{M}_{n}$ or $C^{*} (F)[2] \cong \K$.  Since $C^{*} (F)[2]$ is a non-unital $C^{*}$-algebra ($C^{*} (E)$ is a tight $C^{*}$-algebra over $X_{2}$), we have that $C^{*} (F)[2] \cong \K$.
\end{proof}

\begin{defin}
For a $C^{*}$-algebra $\mathfrak{A}$, set
\begin{align*}
\Sigma \mathfrak{A} = \setof{ x \in K_{0} ( \mathfrak{A} ) }{ \text{$x = [ p ]$ for some projection $p$ in $\mathfrak{A}$}}.
\end{align*}
Let $\mathfrak{B}$ be a $C^{*}$-algebra.  An order isomorphism $\ftn{ \alpha }{ K_{0} ( \mathfrak{A} ) }{ K_{0} ( \mathfrak{B} ) }$ is \emph{scale preserving} if one of the following holds:
\begin{itemize}
\item[(1)] $\mathfrak{A}$ is unital if and only if $\mathfrak{B}$ unital and $\alpha ( [ 1_{ \mathfrak{A} } ] ) = [ 1_{ \mathfrak{B} } ]$.

\item[(2)] $\mathfrak{A}$ is non-unital if and only if $\mathfrak{B}$ is non-unital and $\alpha ( \Sigma \mathfrak{A} ) = \Sigma \mathfrak{B}$.
\end{itemize}
\end{defin}

\begin{theor}\label{t:class1ideal}
Let $E_{1}$ and $E_{2}$ be graphs with finitely many vertices and $C^{*} ( E_{i} )$ is a tight $C^{*}$-algebra over $X_{2}$.  If $\ftn{ \alpha }{ K_{X_{2}}^{+} ( C^{*} ( E_{1} ) ) }{ K_{X_{2}}^{+} ( C^{*} ( E_{2} ) ) }$ is an isomorphism such that $\alpha_{Y}$ is scale preserving for all $Y \in \mathbb{LC} ( X_{2} )$, then there exists an isomorphism $\ftn{ \phi }{ C^{*} ( E_{1} ) }{ C^{*} ( E_{2} ) }$ such that $K_{X_{2}} ( \phi ) = \alpha$.
\end{theor}

\begin{proof}
Since $E_{i}$ has finitely many vertices, $C^{*} ( E_{1} )$ and $C^{*} ( E_{2} )$ are unital $C^{*}$-algebras. 

\medskip

\noindent \emph{Case 1:  Suppose $C^{*} ( E_{1} )$ is an AF-algebra.}  Then $C^{*} ( E_{2} )$ is an AF-algebra.  Hence, the result follows from Elliott's classification of AF-algebras \cite{af}.

\medskip

\noindent \emph{Case 2:  Suppose $C^{*} ( E_{1} )$ is not an AF-algebra.}  Then $C^{*} ( E_{2} )$ is not an AF-algebra. 

\emph{Subcase 2.1: Suppose $C^{*} ( E_{1} ) [1]$ is an AF-algebra.}  Then $C^{*} ( E_{2} ) [ 1 ]$ is an AF-algebra.  By Corollary~\ref{c:unitclassmixed1} and Corollary~\ref{c:uctX}, there exists an isomorphism $\ftn{ \phi }{ C^{*} ( E_{1} ) }{ C^{*} ( E_{2} ) }$ such that $K_{X_{2}} ( \phi ) = \alpha$.

\medskip

\emph{Subcase 2.2: Suppose $C^{*} ( E_{1} )[ 1 ]$ is a Kirchberg algebra.}  Then $C^{*} ( E_{2} ) [ 1 ]$ is a Kirchberg algebra.  Since $C^{*} ( E_{i} )$ is not an AF-algebra, either $C^{*} ( E_{i}) [2]$ is Kirchberg algebra or an AF-algebra.  

Suppose $C^{*} ( E_{i} ) [ 2 ]$ is a Kirchberg algebra.  By Theorem~2.4 of \cite{rr_cexpure}, there exists an isomorphism $\ftn{ \phi }{ C^{*} ( E_{1} ) }{ C^{*} ( E_{2} ) }$ such that $K_{X_{2}} ( \phi ) = \alpha$.  Suppose $C^{*} ( E_{i} ) [2]$ is an AF-algebra.  Then, by Lemma~\ref{l:finitegraphaf}, $C^{*} ( E_{i} ) [ 2 ] \cong \K$.  By Corollary~\ref{c:classmixed2lookslikegraph} and Corollary~\ref{c:uctX}, there exists an isomorphism $\ftn{ \phi }{ C^{*} ( E_{1} ) }{ C^{*} ( E_{2} ) }$ such that $K_{X_{2}} ( \phi ) = \alpha$.
\end{proof}

The following theorem completes the classification of graph $C^{*}$-algebras with exactly one non-trivial ideal.

\begin{corol}\label{c:class1ideal}
Let $E_{1}$ and $E_{2}$ be graphs such that $C^{*} ( E_{i} )$ is a tight $C^{*}$-algebra over $X_{2}$.  Then $C^{*} (E_{1} ) \cong C^{*} ( E_{2} )$ if and only if there exists an isomorphism $\ftn{ \alpha }{ K_{X_{2}}^{+} ( C^{*} ( E_{1} ) ) }{ K_{X_{2}}^{+} ( C^{*} ( E_{2} ) ) }$ such that $\alpha_{Y}$ is a scale preserving isomorphism for all $Y \in \mathbb{LC} ( X_{2} )$.  
\end{corol}

\begin{proof}
The only case that is not covered by Theorem~4.9 of \cite{err:fullext} is the case that $C^{*} ( E_{i} )$ is unital.  The unital case follows from Theorem~\ref{t:class1ideal} because of Theorem \ref{t:metathm}.
\end{proof}

\subsection{Classification of graph $C^{*}$-algebras with more than one ideal}
For a tight $C^{*}$-algebra $\mathfrak{A}$ over $X_{n}$, the finite and infinite simple sub-quotients of $\mathfrak{A}$ are \emph{separated} if there exists $U \in \mathbb{O} ( X_{n} )$ such that either 
\begin{itemize}
\item[(1)] $\mathfrak{A} (U)$ is an AF-algebra and $\mathfrak{A} ( X_{n} \setminus U) \otimes \mathcal{O}_{\infty} \cong \mathfrak{A}( X_{n} \setminus U )$ or

\item[(2)] $\mathfrak{A} (X_{n} \setminus U)$ is an AF-algebra and $\mathfrak{A} ( U) \otimes \mathcal{O}_{\infty} \cong \mathfrak{A}( U )$.
\end{itemize}  In \cite{segrer:ccfis}, the authors proved that if $\mathfrak{A}_{1}$ and $\mathfrak{A}_{2}$ are graph $C^{*}$-algebras that are tight $C^{*}$-algebras over $X_{n}$ such that the finite and infinite simple sub-quotients are separated, then $\mathfrak{A}_{1} \otimes \K \cong \mathfrak{A}_{2} \otimes \K$ if and only if $K_{X_{n}}^{+} ( \mathfrak{A}_{1} ) \cong K_{X_{n}}^{+} ( \mathfrak{A}_{2} )$.  
We will show in this section that under mild $K$-theoretical conditions, we may remove the separated condition for the case $n = 3$.

\begin{lemma}\label{l:fullmixed}
Let $E$ be a graph such that $C^{*} ( E )$ is a tight $C^{*}$-algebra over $X_{n}$.
\begin{itemize}
\item[(i)] If $C^{*} ( E )[n]$ and $C^{*} ( E ) [ 1 ]$ are purely infinite and $C^{*} ( E ) [ 2, n-1]$ is an AF-algebra, then 
\begin{align*}
\mathfrak{e}_{1} : 0 \to C^{*} ( E ) [ 2, n ] \otimes \K \to C^{*} ( E ) \otimes \K \to C^{*} ( E ) [ 1 ] \otimes \K \to 0
\end{align*}
is a full extension.

\item[(ii)] If $C^{*} (E)[ k, n ]$ and $C^{*} (E)[ 1, k-2 ]$ are AF-algebras and $C^{*} ( E ) [k-1]$ is purely infinite, then 
\begin{align*}
\mathfrak{e}_{2} : 0 \to C^{*} ( E ) [ k , n ] \otimes \K \to C^{*} ( E ) \otimes \K \to C^{*} ( E ) [ 1, k-1 ] \otimes \K \to 0
\end{align*}
is a full extension.
\end{itemize}
\end{lemma}

\begin{proof}
Suppose $C^{*} ( E )[n]$ and $C^{*} ( E ) [ 1 ]$ are purely infinite and $C^{*} ( E ) [ 2, n-1]$ is an AF-algebra.  Note that $C^{*}( E ) [ 1, n - 1 ] / C^{*} ( E ) [ 2, n-1 ] \cong C^{*} ( E ) [ 1 ]$ and $C^{*} ( E ) [ 2, n-1 ]$ is the largest ideal of $C^{*} ( E )[1,n-1]$ which is an AF-algebra.  Since $C^{*} ( E ) [1, n-1]$ is isomorphic to a graph $C^{*}$-algebra, by Proposition~3.10 of \cite{semt_classgraphalg},
\begin{align*}
0 \to C^{*} ( E ) [2, n-1 ] \otimes \K \to C^{*} ( E ) [ 1, n-1 ] \otimes \K \to C^{*} ( E )[1] \otimes \K \to 0
\end{align*}
is a full extension.  Since $C^{*} ( E ) [ n ] \otimes \K$ is a purely infinite simple $C^{*}$-algebra, we have that 
\begin{align*}
0 \to C^{*} ( E ) [ n ] \otimes \K \to C^{*} ( E ) [ 2,n] \otimes \K \to C^{*} ( E ) [ 2, n-1 ] \otimes \K \to 0
\end{align*}
is a full extension.  Hence, by Proposition~3.2 of \cite{ERRlinear}, $\mathfrak{e}_{1}$ is a full extension.

Suppose $C^{*} (E)[ k, n ]$ and $C^{*} (E)[ 1, k-2 ]$ are AF-algebras and $C^{*} ( E ) [k-1]$ is purely infinite.  Note that $C^{*} ( E )[k,n]$ is the largest ideal of $C^{*} ( E ) [ k -1, n ]$ such that $C^{*} ( E ) [ k  , n ]$ is an AF-algebra and $C^{*} (E) [ k-1,n ] / C^{*} ( E ) [k, n ] \cong C^{*} ( E ) [ k-1]$ is purely infinite.  Since $C^{*} ( E ) [ k - 1, n ] \otimes \K$ is isomorphic to a graph $C^{*}$-algebra, by Proposition~3.10 of \cite{semt_classgraphalg}, 
\begin{align*}
0 \to C^{*} ( E )[k, n ] \otimes \K \to C^{*} ( E ) [ k-1, n ] \otimes \K \to C^{*} ( E ) [ k-1 ] \otimes \K \to 0
\end{align*}
is a full extension.  By Proposition~5.4 of \cite{segrer:ccfis}, $\mathfrak{e}_{2}$ is a full extension.
\end{proof}

\begin{theor}\label{t:graphmixed1}
Let $E_{1}$ and $E_{2}$ be graphs such that $C^{*} ( E_{i} )$ is a tight $C^{*}$-algebra over $X_{n}$.  Suppose
\begin{itemize}
\item[(i)] $C^{*} ( E_{i} )[ n ]$ and $C^{*} ( E_{i} ) [ 1 ]$ are purely infinite;

\item[(ii)] $C^{*} ( E_{i} ) [ 2, n-1]$ is an AF-algebra; and

\item[(iii)] $\kk^{1} ( C^{*} ( E_{1} )[1] ,  C^{*} ( E_{2} ) [2,n] ) = \kl^{1}( C^{*} ( E_{1} )[1] ,  C^{*} ( E_{2} ) [2,n] )$.
\end{itemize}
Then $C^{*} ( E_{1} ) \otimes \K \cong C^{*} ( E_{2} ) \otimes \K$ if and only if $K_{X_{n}}^{+} ( C^{*} (E_{1}) \otimes \K) \cong K_{X_{n} }^{+} ( C^{*} ( E_{2} ) \otimes \K )$.
\end{theor}

\begin{proof}
Let $\mathfrak{e}_{i}$ be the extension 
\begin{align*}
0 \to C^{*} ( E_{i} ) [ 2, n] \otimes \K \to C^{*} ( E_{i} ) \otimes \K \to C^{*} ( E_{i} )[1] \otimes \K \to 0.
\end{align*} 
By Lemma~\ref{l:fullmixed}(i), $\mathfrak{e}_{i}$ is a full extension.  Suppose $\ftn{ \alpha }{ K_{X_{n}}^{+} ( C^{*} ( E_{1} ) \otimes \K ) }{ K_{X_{n}}^{+} ( C^{*} ( E_{2} )  \otimes \K ) }$.  Lift $\alpha$ to an invertible element $x \in \kk ( X_{n} ; C^{*} ( E_{1} ) \otimes \K , C^{*} ( E_{2} ) \otimes \K )$.  Note that $r_{X_{n}}^{[2,n]} (x)$ is invertible in $\kk( [2,n] ; C^{*} ( E_{1} ) [ 2,n] \otimes \K , C^{*} ( E_{2} )[2, n ] \otimes \K )$ and $r_{X_{n}}^{[1]} (x)$ is invertible in $\kk ( C^{*} ( E_{1} )[1] \otimes \K , C^{*} ( E_{2} )[1] \otimes \K )$.  By Theorem~\ref{t:classmixed1}, there exists an isomorphism $\ftn{ \phi_{0} }{ C^{*} ( E_{1} ) [ 2, n ] \otimes \K }{ C^{*} ( E_{2} ) [ 2, n ] \otimes \K }$ such that $\kl ( \phi_{0} ) = z$, where $z$ is the invertible element of $\kl( C^{*} ( E_{1} ) [ 2,n] \otimes \K , C^{*} ( E_{2} )[2, n ] \otimes \K )$ induced by $r_{X_{n}}^{[2,n]} (x)$.  By the Kirchberg-Phillips classification (\cite{kirchpure} and \cite{phillipspureinf}), there exists an isomorphism $\ftn{ \phi_{2} }{ C^{*} ( E_{1} ) [1 ] \otimes \K }{ C^{*} ( E_{2} ) [1] \otimes \K }$  such that $\kk ( \phi_{2} ) = r_{X_{n}}^{[1]} (x)$.  

Consider $C^{*} ( E_{i} )$ as a $C^{*}$-algebra over $X_{2}$ by setting $C^{*} ( E_{i} ) [ 2 ] = C^{*} ( E_{i} ) [ 2,n]$ and $C^{*} ( E_{i} ) [ 1,2 ] = C^{*} ( E_{i} )$.  Let $y$ be the invertible element in $\kk ( X_{2} , C^{*} ( E_{1} ) , C^{*} ( E_{2} ) )$ induced by $x$.  Note that $r_{X_{2}}^{[1]} (y) = r_{X_{n}}^{[1]} (x) = \kk ( \phi_{2} )$ and $\kl ( r_{X_{2}}^{[2]} (y) ) = z = \kl ( \phi_{0} )$ in $\kl ( C^{*} ( E_{1} )[2,n] , C^{*} ( E_{2} )[2,n] )$.  By Theorem~3.7 of \cite{segrer:ccfis}, 
\begin{align*}
r_{X_{2}}^{[1]} (y) \times [ \tau_{ \mathfrak{e}_{2} } ] = [ \tau_{ \mathfrak{e}_{1} } ] \times r_{X_{2}}^{[2]} (y) 
\end{align*}    
in $\kk^{1} ( C^{*} ( E_{1} ) [1]  \otimes \K , C^{*} ( E_{2} ) [ 2,n] \otimes \K )$, where $\mathfrak{e}_{i}$ is the extension
\begin{align*}
0 \to C^{*} ( E_{i} ) [2,n ] \otimes \K \to C^{*} ( E_{i} ) \otimes \K \to C^{*} ( E_{i} ) [ 1 ] \otimes \K \to 0.
\end{align*}
Thus,
\begin{align*}
\kl ( \phi_{2} ) \times  [ \tau_{ \mathfrak{e}_{2} } ] = [ \tau_{ \mathfrak{e}_{1} } ] \times \kl ( \phi_{0} )
\end{align*}   
in $\kl^{1} ( C^{*} ( E_{1} ) [1]  \otimes \K , C^{*} ( E_{2} ) [ 2,n]  \otimes \K )$.  Since $\kl^{1}( C^{*} ( E_{1} )[1] \otimes \K,  C^{*} ( E_{2} ) [2,n]  \otimes \K) = \kk^{1} ( C^{*} ( E_{1} )[1] \otimes \K,  C^{*} ( E_{2} ) [2,n]  \otimes \K)$, 
\begin{align*}
\kk ( \phi_{2} ) \times  [ \tau_{ \mathfrak{e}_{2} } ] = [ \tau_{ \mathfrak{e}_{1} } ] \times \kk ( \phi_{0} )
\end{align*}   
in $\kk^{1} ( C^{*} ( E_{1} ) [1] \otimes \K  , C^{*} ( E_{2} ) [ 2,n]  \otimes \K )$.  By Lemma~4.5 of \cite{segrer:ccfis}, $C^{*} ( E_{1} ) \otimes \K \cong C^{*} ( E_{2} ) \otimes \K$.
\end{proof}

\begin{theor}\label{t:graphmixed2}
Let $E_{1}$ and $E_{2}$ be graphs such that $C^{*} ( E_{i} )$ is a tight $C^{*}$-algebra over $X_{n}$.  Suppose
\begin{itemize}
\item[(i)] $C^{*} ( E_{i} )[ k,n]$ and $C^{*} ( E_{i} ) [ 1, k-2 ]$ are AF-algebras;

\item[(ii)] $C^{*} ( E_{i} ) [ k-1]$ is purely infinite; and

\item[(iii)] $\kk^{1} ( C^{*} ( E_{1} )[1, k-1] ,  C^{*} ( E_{2} ) [k,n] ) = \kl^{1}( C^{*} ( E_{1} )[1,k-1] ,  C^{*} ( E_{2} ) [k,n] )$.
\end{itemize}
Then $C^{*} ( E_{1} ) \otimes \K \cong C^{*} ( E_{2} ) \otimes \K$ if and only if $K_{X_{n}}^{+} ( C^{*} (E_{1}) \otimes \K ) \cong K_{X_{n} }^{+} ( C^{*} ( E_{2} ) \otimes \K)$.  
\end{theor}

\begin{proof}
Let $\mathfrak{e}_{i}$ be the extension $0 \to C^{*} ( E_{i} ) [ k, n ] \otimes \K \to C^{*} ( E_{i} ) \otimes \K \to C^{*} ( E_{i} ) [ 1, k-1 ] \otimes \K \to 0$.  By Lemma~\ref{l:fullmixed}(ii), $\mathfrak{e}_{i}$ is a full extension.  Suppose $\ftn{ \alpha }{ K_{X_{n}}^{+} ( C^{*} ( E_{1} ) \otimes \K) }{ K_{X_{n}}^{+} ( C^{*} ( E_{2} ) \otimes \K ) }$.  Lift $\alpha$ to an invertible element $x \in \kk ( X_{n} ; C^{*} ( E_{1} ) \otimes \K, C^{*} ( E_{2} ) \otimes \K )$.  Note that $r_{X_{n}}^{ [k,n] } (x)$ is invertible in $\kk ( [k, n]; C^{*} ( E_{1} ) [ k,n]  \otimes \K, C^{*} ( E_{2} ) [ k , n ] \otimes \K  )$ and $r_{X_{n}}^{[1,k-1]} (x)$ is invertible in $\kk (C^{*} ( E_{1} ) [ 1, k-1 ] , C^{*} ( E_{2} ) [ 1, k-1 ] )$.  By Theorem~\ref{t:classmixed1}, there exists an isomorphism $\ftn{ \phi_{2} }{ C^{*} ( E_{1} ) [ 1, k-1 ] \otimes \K }{ C^{*} ( E_{2 } )[ 1, k-1] \otimes \K }$ such that $\kl ( \phi_{2} ) = z_{2}$, where $z_{2}$ is the invertible element in $\kl (C^{*} ( E_{1} ) [ 1, k-1 ] , C^{*} ( E_{2} ) [ 1, k-1 ] )$ induced by $r_{X_{n}}^{[1,k-1]} (x)$.  By Elliott's classification \cite{af}, there exists an isomorphism $\ftn{ \phi_{0} }{ C^{*} ( E_{1} ) [ k, n ] \otimes \K }{ C^{*} ( E_{2} ) [ k , n ] \otimes \K }$ such that $\kk ( \phi_{0} ) = z_{0}$, where $z_{0}$ is the invertible element in $\kk ( C^{*} ( E_{1} ) [ k,n]  \otimes \K, C^{*} ( E_{2} ) [ k , n ] \otimes \K  )$ induced by $r_{X_{n}}^{ [k,n] } (x)$.

Consider $C^{*} ( E_{i} )$ as a $C^{*}$-algebra over $X_{2}$ by setting $C^{*} ( E_{i} ) [ 2 ] = C^{*} ( E_{i} ) [ k,n]$ and $C^{*} ( E_{i} ) [ 1,2 ] = C^{*} ( E_{i} )$.  Let $y$ be the invertible element in $\kk ( X_{2} , C^{*} ( E_{1} ) , C^{*} ( E_{2} ) )$ induced by $x$.  Note that $\kl ( r_{X_{2}}^{[1]} (y) ) = z_{2} = \kl ( \phi_{2} )$ and $r_{X_{2}}^{[2]} (y)  = z_{0} = \kk ( \phi_{0} )$.  By Theorem~3.7 of \cite{segrer:ccfis}, 
\begin{align*}
r_{X_{2}}^{[1]} (y) \times [ \tau_{ \mathfrak{e}_{2} } ] = [ \tau_{ \mathfrak{e}_{1} } ] \times r_{X_{2}}^{[2]} (y) 
\end{align*}    
in $\kk^{1} ( C^{*} ( E_{1} ) [1,k-1] \otimes \K , C^{*} ( E_{2} ) [ k,n] \otimes \K )$, where $\mathfrak{e}_{i}$ is the extension
\begin{align*}
0 \to C^{*} ( E_{i} ) [ k,n] \otimes \K \to C^{*} ( E_{i} ) \otimes \K \to C^{*} ( E_{i} ) [ 1,k-1] \otimes \K \to 0.
\end{align*}
Thus, 
\begin{align*}
\kl ( \phi_{2} ) \times  [ \tau_{ \mathfrak{e}_{2} } ] = [ \tau_{ \mathfrak{e}_{1} } ] \times \kl ( \phi_{0} )
\end{align*}   
in $\kl^{1} ( C^{*} ( E_{1} ) [1,k-1] \otimes \K , C^{*} ( E_{2} ) [ k,n] \otimes \K )$.  Since $\kl^{1}( C^{*} ( E_{1} )[1,k-1] \otimes \K,  C^{*} ( E_{2} ) [k,n]  \otimes \K) = \kk^{1} ( C^{*} ( E_{1} )[1,k-1] \otimes \K,  C^{*} ( E_{2} ) [k,n]  \otimes \K)$, 
\begin{align*}
\kk ( \phi_{2} ) \times  [ \tau_{ \mathfrak{e}_{2} } ] = [ \tau_{ \mathfrak{e}_{1} } ] \times \kk ( \phi_{0} )
\end{align*}   
in $\kk^{1} ( C^{*} ( E_{1} ) [1,k-1] \otimes \K , C^{*} ( E_{2} ) [ k,n] \otimes \K)$.  By Lemma~4.5 of \cite{segrer:ccfis}, $C^{*} ( E_{1} ) \otimes \K \cong C^{*} ( E_{2} ) \otimes \K$.
\end{proof}

\begin{theor}\label{t:graphmixed3}
Let $E_{1}$ and $E_{2}$ be graphs such that $C^{*} ( E_{i} )$ is a tight $C^{*}$-algebra over $X_{3}$.  Suppose $K_{0} ( C^{*} (E_{1})[1] )$ is the direct sum of cyclic groups if $C^{*} ( E_{1} ) [1]$ is purely infinite and $K_{0} ( C^{*} ( E_{1} ) [ 1,2] )$ is the direct sum of cyclic groups if $C^{*} ( E_{1} ) [1]$ is an AF-algebra.  Then $C^{*} ( E_{1} ) \otimes \K \cong C^{*} ( E_{2} ) \otimes \K$ if and only if $K_{X_{3}}^{+} ( C^{*} ( E_{1} ) ) \cong K_{X_{3}}^{+} ( C^{*} ( E_{2} ) )$.
\end{theor}

\begin{proof}
The ``only if'' direction is clear.  Suppose $K_{X_{3}}^{+} ( C^{*} ( E_{1} ) ) \cong K_{X_{3}}^{+} ( C^{*} ( E_{2} ) )$.  Suppose $C^{*} ( E_{1})[1]$ is purely infinite.  Then $K_{0} ( C^{*}(E_{1} )[1] )$ is the direct sum of cyclic groups.  Thus, $\mathrm{Pext}_{\Z}^{1} ( K_{0} ( C^{*} ( E_{1} )[1] ) , K_{0} ( C^{*} ( E_{2} ) [2] ) ) = 0$.  Since $K_{1} ( C^{*} ( E_{1} )[1] )$ is a free group, $\mathrm{Pext}_{\Z}^{1} ( K_{1} ( C^{*} ( E_{1} )[1] ) , K_{1} ( C^{*} ( E_{2} ) [2] ) ) = 0$.  Hence,  
\begin{align*}
\kk^{1} ( C^{*} ( E_{1} ) [1] , C^{*} ( E_{2} ) [2,3] ) = \kl^{1} ( C^{*} ( E_{1} ) [1] , C^{*} ( E_{2} ) [2,3] ).
\end{align*}

Suppose $C^{*} ( E_{1})[1]$ is an AF-algebra.  Then $K_{0} ( C^{*}(E_{1} )[1,2] )$ is the direct sum of cyclic groups.  Thus, $\mathrm{Pext}_{\Z}^{1} ( K_{0} ( C^{*} ( E_{1} )[1,2] ) , K_{0} ( C^{*} ( E_{2} ) [3] ) ) = 0$.  Since $K_{1} ( C^{*} ( E_{1} )[1,2] )$ is a free group, $\mathrm{Pext}_{\Z}^{1} ( K_{1} ( C^{*} ( E_{1} )[1,2] ) , K_{1} ( C^{*} ( E_{2} ) [3] ) )  = 0$.   Therefore, 
\begin{align*}
\kk^{1} ( C^{*} ( E_{1} ) [1,2] , C^{*} ( E_{2} ) [3] ) = \kl^{1} ( C^{*} ( E_{1} ) [1,2] , C^{*} ( E_{2} ) [3] ).
\end{align*}

\medskip

\noindent \emph{Case 1: Suppose the finite and infinite simple sub-quotients of $C^{*} ( E_{1} )$ are separated.} Then the finite and infinite simple sub-quotients of $C^{*} ( E_{2} )$ are separated.  Hence, by Theorem~6.9 of \cite{segrer:ccfis}, $C^{*} ( E_{1} ) \otimes \K \cong C^{*} ( E_{2} ) \otimes \K$.  

\medskip

\noindent \emph{Case 2: Suppose the finite and infinite simple sub-quotients of $C^{*} ( E_{1} )$ are not separated.}  Then the finite and infinite simple sub-quotients of $C^{*} ( E_{2} )$ are not separated.    

\medskip

\emph{Subcase 2.1:  Suppose $C^{*} ( E_{1} ) [3]$ and $C^{*} (E_{1})[1]$ are purely infinite and $C^{*} ( E_{1} ) [2]$ is an AF-algebra.}  Then $C^{*} ( E_{2} ) [3]$ and $C^{*} (E_{2})[1]$ are purely infinite and $C^{*} ( E_{2} ) [2]$ is an AF-algebra.  Then by the above paragraph we have that $\kk^{1} ( C^{*} ( E_{1} ) [1] , C^{*} ( E_{2} ) [2,3] ) = \kl^{1} ( C^{*} ( E_{1} ) [1] , C^{*} ( E_{2} ) [2,3] )$.  Hence, by Theorem~\ref{t:graphmixed1}, $C^{*} ( E_{1} ) \otimes \K \cong C^{*} ( E_{2} ) \otimes \K$.

\medskip

\emph{Subcase 2.2: Suppose $C^{*} ( E_{1} ) [3]$ and $C^{*} (E_{1})[1]$ are AF-algebras and $C^{*} ( E_{1} ) [2]$ is purely infinite.}  Then $C^{*} ( E_{2} ) [3]$ and $C^{*} (E_{2})[1]$ are AF-algebras and $C^{*} ( E_{2} ) [2]$ is purely infinite.  Then by the above paragraph we have that 
\begin{align*}
\kk^{1} ( C^{*} ( E_{1} ) [1,2] , C^{*} ( E_{2} ) [3] ) = \kl^{1} ( C^{*} ( E_{1} ) [1,2] , C^{*} ( E_{2} ) [3] ).
\end{align*}
Hence, by Theorem~\ref{t:graphmixed2}, $C^{*} ( E_{1} ) \otimes \K \cong C^{*} ( E_{2} ) \otimes \K$.
\end{proof}

\begin{corol}
Let $E_{1}$ and $E_{2}$ be graphs such that $C^{*} ( E_{i} )$ is a tight $C^{*}$-algebra over $X_{3}$.  Suppose that $K_{0} ( C^{*} (E_{i}) )$ is finitely generated.  Then $C^{*} ( E_{1} ) \otimes \K \cong C^{*} ( E_{2} ) \otimes \K$ if and only if $K_{X_{3}}^{+} ( C^{*} ( E_{1} ) ) \cong K_{X_{3}}^{+} ( C^{*} ( E_{2} ) )$.
\end{corol}

\begin{proof}
Since $C^{*} (E_{i})$ is real rank zero, the canonical projection $\ftn{ \pi }{ C^{*} (E_{i} ) } { C^{*} ( E_{i} ) [ 1 ] }$ induces a surjective homomorphism $\ftn{ \pi }{ K_{0} ( C^{*} (E_{i} ) ) } { K_{0} ( C^{*} ( E_{i} ) [ 1 ] ) }$.  Hence, $K_{0} ( C^{*} ( E_{i} ) [ 1 ] )$ is finitely generated since $K_{0} ( C^{*} (E_{i}) )$ is finitely generated.  The corollary now follows from Theorem~\ref{t:graphmixed3}.
\end{proof}

\def\cprime{$'$}

%\bibliographystyle{siam}
%\bibliography{/Users/ruize/Math/reference/references}

\end{document}